%% file: main.tex
\documentclass[final]{siamart190516}

\usepackage{amsmath, amsbsy, amsgen, amssymb, bm, mathtools}
\usepackage{graphicx, hyperref}
\usepackage[noend]{algpseudocode}
\usepackage{algorithm, algorithmicx}
\usepackage{subcaption, float}
\usepackage{makecell}
\usepackage{multirow}
\usepackage{dsfont}
\usepackage{booktabs}
\usepackage{subcaption}
\usepackage[normalem]{ulem}

\newcommand{\ignore}[1]{}

\usepackage{tikz}
\usetikzlibrary{shapes.geometric, arrows.meta, positioning, calc}
\tikzset{
  startstop/.style = {ellipse, draw, very thick, minimum width=3cm, minimum height=1cm, align=center},
  process/.style   = {rectangle, draw, very thick, minimum width=3.2cm, minimum height=1cm, align=center},
  io/.style        = {trapezium, trapezium left angle=70, trapezium right angle=110, draw, very thick, minimum width=3cm, minimum height=1cm, align=center},
  decision/.style  = {diamond, aspect=2, draw, very thick, minimum width=3.2cm, minimum height=1cm, align=center},
  arrow/.style     = {->, >=Stealth, thick, shorten >=3pt, shorten <=3pt},
  merge/.style     = {coordinate}
}

\algrenewcommand\alglinenumber[1]{\sf\scriptsize\color{black}{#1}}
\algrenewcommand\algorithmicrequire{\textbf{Require:}}
\algrenewcommand\algorithmicensure{\textbf{Ensure:}}

\numberwithin{equation}{section}
\numberwithin{theorem}{section}
\numberwithin{figure}{section}
\newtheorem{remark}{Remark}[section]

\newcommand{\vct}[1]{\bm{#1}}
\newcommand{\mtx}[1]{\bm{#1}}

\newcommand{\jh}[1]{\ifcomment \textcolor{red}{[\textbf{JH:} #1]}\fi}
\newcommand{\yn}[1]{\ifcomment \textcolor{blue}{[\textbf{YN:} #1]}\fi}

\newif\ifcomment
\commentfalse
\newif\ifversionone
\versiononefalse
\newif\ifversiontwo
\versiontwotrue

\newif\iftransfer
\transferfalse   

\def\mainalg{APLICUR}
\def\F{\mathbb{R}}

\def\l{\ell}
\def\Vhat{\mtx{\widehat{U}}}
\def\Vhat{\mtx{\widehat{V}}}
\def\Lhat{\mtx{\widehat{\Sigma}}}
\def\svhat{\hat\sigma}

\title{Adaptive LSQR Preconditioning\\from One Small Sketch}
\author{Jung Eun Huh, Coralia Cartis, Yuji Nakatsukasa}
\date{March 2025}

\begin{document}

\maketitle

\begin{abstract}
    We propose \mainalg{}, an adaptive preconditioning framework for large-scale linear least-squares (LLS) problems. Using a single small sketch computed once at initialization, \mainalg{} incrementally refines a CUR-based preconditioner throughout the Krylov solve, interleaving preconditioning with iteration. This enables early convergence without the need to construct a costly high-quality preconditioner upfront.
    With a modest sketch dimension (typically $5 - 250$), largely
    independent of both the problem size and numerical rank, \mainalg{} achieves convergence guarantees that are likewise independent of the sketch size. The method is applicable to general matrices without structural assumptions (e.g. need not be heavily overdetermined) and is well suited to large, sparse, or numerically low-rank problems.
    We conduct extensive numerical studies to examine the behavior of the proposed framework and guide the effective algorithmic design choices. Across a range of test problems, \mainalg{} achieves competitive or improved time-to-accuracy performance compared with established randomized preconditioners, including Blendenpik and Nystr\"om PCG, while maintaining low setup cost and robustness across problem regimes.
\end{abstract}

\section{Introduction}

Large-scale regularized linear least-squares (LLS) of the form
\begin{align}\label{eq:regularizedLLS}
    \min_{\vct{x} \in \F^n} \|\mtx{A}\vct{x}-\vct{b}\|_2^2 + \mu \|\vct{x}\|_2^2, \quad \mu \geq 0,
\end{align}
with $\mtx{A} \in \F^{m\times n}$, $\vct{b} \in \F^m$, and $m\geq n\geq1$ arise throughout scientific computing, inverse problems, and machine learning. For large and potentially ill-conditioned systems, Krylov subspace methods such as CG~\cite{hestenes1952methods} applied to the normal equation and LSQR~\cite{paige1982lsqr} are widely used for solving \eqref{eq:regularizedLLS} due to their low memory footprint and reliance only on matrix–vector products. However, their convergence can be prohibitively slow when the system is ill-conditioned or has unfavorable spectral structure~\cite{barthelme2021spectral}.

Preconditioning is a fundamental tool for accelerating Krylov subspace methods for large-scale linear systems and least-squares problems~\cite{benzi2002preconditioning}. Classical approaches aim to improve conditioning with deterministic techniques including polynomial preconditioning~\cite{ashby1988polynomial} and incomplete factorization methods~\cite{benzi2003robust}. While effective in many settings, such methods can be expensive to adapt to large-scale problems.
Randomized sketch-based approaches have achieved significant success by efficiently constructing a preconditioner from a spectral surrogate of the system matrix using a randomized embedding~\cite{halko2011finding,drineas2016randnla,mahoney2011randomized}. For example, Blendenpik~\cite{avron2010blendenpik, rokhlin2008fast} uses a sketched QR factorization to build a full-rank preconditioner, while LSRN~\cite{meng2014lsrn} employs random normal projections to construct a well-conditioned LLS.  Blendenpik-variant SKiLLS \cite{cartis2021hashing}  exploits sparse sketching for sparse and dense, possibly rank-deficient, inputs. 

Recent line of research based on randomized Nystr\"om approximation~\cite{frangella2023randomized,derezinski2025faster,zhao2024adaptive} construct low-rank spectral preconditioner adapted to the numerical rank. A key observation underlying these approaches is that matrices arising in applications often exhibit rapidly decaying or clustered spectra, particularly in kernel methods and Gaussian process settings~\cite{alaoui2015fast,gonen2016solving,barthelme2021spectral,daskalakis2022good}. This structure enables accurate low-rank approximation and suggests that effective preconditioning need not reshape the entire spectrum, but rather should target the dominant spectral components. This perspective is closely related to deflation and augmentation techniques in Krylov methods~\cite{gutknecht2012spectral,gaul2013framework,coulaud2013deflation}, which accelerate convergence by removing or correcting a troublesome subspace.

Despite this progress, existing randomized preconditioners exhibit several limitations. Many methods require sketch sizes that scale with the problem dimension or numerical rank, or rely on repeated sketching to estimate this rank~\cite{avron2010blendenpik,lacotte2020effective,diaz2023robust,frangella2023randomized,zhao2024adaptive,derezinski2024solving}. Moreover, most approaches construct a \emph{static} preconditioner prior to the Krylov solve, which can incur significant upfront cost and delay convergence, particularly when only moderate accuracy is required. 
A recent work~\cite{derezinski2025faster} theoretically explores an adaptive and multi-level sketching strategy for preconditioning, but it still relies on repeated sketching, which requires multiple passes over the matrix, and the sketch size depends on the desired preconditioning strength.
In addition, many methods rely on restrictive structural assumptions, such as highly overdetermined regimes~\cite{avron2010blendenpik}, or require symmetry or positive definiteness---often enforced by working with the normal equations~\cite{cutajar2016preconditioning,derezinski2025faster,diaz2023robust,frangella2023randomized}, which may be unsuitable or numerically unstable for general rectangular problems. In particular, the normal equations approach is 
to be avoided unless the condition number $\kappa(\mtx{A})$ is $O(1)$ in double precision, to avoid numerical instability~\cite[Section 5.3.7]{golub2013matrix}. 

These limitations motivate the development of adaptive preconditioning strategies that better align computational effort with the evolving needs of the solver. These limitations motivate the development of adaptive preconditioning strategies that better align computational effort with the evolving needs of the solver.
Consequently, designing preconditioners that simultaneously (i) avoid large or repeated sketches, (ii) adapt to unknown spectral structure, (iii) integrates naturally with the evolving needs of the Krylov solver, and (iv) apply broadly to general matrices without structural assumptions remains an open challenge.

\subsection{Contributions}

In this work, we address these challenges by introducing \mainalg{}, an adaptively preconditioned LSQR method based on iterative CUR approximations. The key idea is to begin solving immediately using a cheap, low-rank preconditioner, and to refine this preconditioner only when additional spectral correction is needed. Rather than constructing a high-quality preconditioner upfront, \mainalg{} alternates between short LSQR phases and incremental updates of a CUR-based spectral preconditioner. Crucially, all updates are driven by a single small sketch, computed once at initialization and reused throughout.

Our approach combines two key ingredients. First, we employ CUR-based low-rank approximations~\cite{osinsky2025close,zamarashkin2018existence,pritchard2025fast}, which select actual rows and columns of $\mtx{A}$ as natural bases for approximation, thereby preserving structural properties such as sparsity and enabling efficient incremental updates. Second, we adopt an adaptive preconditioning strategy that interleaves preconditioner refinement with Krylov iterations, aligning spectral correction with the evolving needs of the solver.

Our main contributions are as follows:

\begin{enumerate}
    \item \emph{Single-sketch adaptive preconditioning.} We propose a framework in which a single small sketch (typically of dimension $5$–$250$), independent of both the matrix dimensions and the unknown numerical rank, supports a fully adaptive preconditioner whose rank grows incrementally during the solve.

    \item \emph{Interleaved preconditioning and solving.} We introduce an adaptive strategy that alternates between LSQR iterations and preconditioner updates, enabling early progress and avoiding the cost of constructing a high-quality preconditioner upfront. This allows computational effort to be allocated adaptively throughout the solve.

    \item \emph{CUR-based adaptive preconditioning with broad applicability.} By leveraging iterative CUR approximations, we construct a low-rank spectral preconditioner that supports efficient rank updates while inheriting structural properties of $\mtx{A}$ such as sparsity. The method applies to general rectangular matrices without symmetry or definiteness assumptions, and does not rely on highly overdetermined regimes ($m \gg n$).
    
    \item \emph{Sketch-size-independent guarantees.} We establish bounds on the condition number and convergence rate of the preconditioned system that depend on the CUR approximation quality, but not on the sketch dimension.
    
    \item \emph{Empirical study and algorithmic design.} Extensive experiments across diverse spectral and structural regimes demonstrate that \mainalg{} achieves rapid early convergence and competitive or improved time-to-accuracy performance compared with Blendenpik and Nystr\"om PCG, particularly on large-scale sparse problems. Targeted numerical studies further provide insight into the algorithm’s behavior and guide the design of an efficient implementation.
\end{enumerate}

\subsection{Roadmap}
The remainder of the paper is organized as follows. \cref{sec:preliminaries} reviews preliminaries and the spectral motivation for preconditioning. \cref{sec:highlevelalg} introduces the CUR-based spectral preconditioner and the adaptive framework. \cref{sec:theory} presents the convergence analysis. \cref{sec:implementation} details the full implementation of \mainalg{}. \cref{sec:expframework,sec:numericalstudies,sec:benchmark} describe the experimental setup and report extensive numerical studies and comparisons with competing methods.

\section{Preliminaries and Motivation}\label{sec:preliminaries}
\subsection{Notations and problem formulation}
The regularized problem in \cref{eq:regularizedLLS} is equivalent to the augmented standard least-squares problem
\begin{align}\label{eq:augLLS}
    \min_{\vct{x}} \|\mtx{A}_\mu \vct{x} - \vct{b_{\text{aug}}}\|_2^2, \quad \mtx{A}_\mu = \begin{bmatrix}\mtx{A}\\ \mu\mtx{I}\end{bmatrix}, \quad \vct{b}_{\text{aug}}=\begin{bmatrix}\vct{b}\\\vct{0}\end{bmatrix}.
\end{align}

Let the singular value decomposition (SVD) of $\mtx{A}$ in \eqref{eq:regularizedLLS} be $$\mtx{A} = \mtx{U}\mtx{\Sigma}\mtx{V}^\top,$$where $\mtx{U} \in \F^{m \times n}$ and $\mtx{V} \in \F^{n\times n}$ are orthonormal and $\mtx{\Sigma} = \text{diag}(\sigma_1, \dots, \sigma_n) \in \F^{n \times n}$ contains the singular values in descending order. The singular values of $\mtx{A}_\mu$ are given by $\sigma_i(\mtx{A}_\mu)=\sqrt{\sigma_i^2 + \mu^2}$.

The optimal rank-$\l$ approximation is given by the \emph{rank-$\ell$ truncated SVD}~\cite{eckart1936approximation}:
$$\lfloor\mtx{A}\rfloor_\ell = \mtx{U}_\ell \mtx{\Sigma}_\ell \mtx{V}_\ell^\top,$$
where $\mtx{U}_\ell$ and $\mtx{V}_\ell$ consist of the first $\ell$ columns of $\mtx{U}$ and $\mtx{V}$, respectively, and $\mtx{\Sigma}_\ell$ is the $\ell \times \ell$ principal submatrix of $\mtx{\Sigma}$.

Throughout this paper, we denote the spectral (or vector $\ell_2$) norm by $\|\cdot\|_2$ and the Frobenius norm by $\|\cdot\|_F$. The $i$-th largest singular value of a matrix $\mtx{M}$ is denoted by $\sigma_i(\mtx{M})$, and $(\cdot)^\dagger$ indicates the Moore-Penrose pseudoinverse. We adopt MATLAB-style indexing for submatrices: for index vectors $\mtx{p}, \mtx{q} \in \F^\ell$, $\mtx{M}(\mtx{p}, :)$ and $\mtx{M}(:, \mtx{q})$ denote the submatrices consisting of rows indexed by $\mtx{p}$ and columns indexed by $\mtx{q}$, respectively.

\subsection{LSQR and convergence}\label{sec:section2}

LSQR~\cite{paige1982lsqr} is a Krylov subspace method for solving least-squares problems, including regularized ones~\cref{eq:regularizedLLS}. For simplicity, assume that $\mu = 0$. At iteration $k$, LSQR computes $\vct{x}_k$ that minimizes the residual norm $\|\mtx{A}\vct{x}-\vct{b}\|_2$ over the affine Krylov subspace
\begin{align*}
    \mathcal{K}_k^{\vct{x}_0} := \vct{x}_0 + \mathcal{K}_k(\mtx{A}^\top \mtx{A}, \mtx{A}^\top \vct{r}_0), 
    \quad \vct{r}_0 = \vct{b} - \mtx{A}\vct{x}_0,
\end{align*}
where $\mathcal{K}_k(\mtx{M},\vct{v}) = \mathrm{span}\{\vct{v},\mtx{M}\vct{v},\ldots,\mtx{M}^{k-1}\vct{v}\}$.

The residual norm decomposes orthogonally into two components,
\[
\|\mtx{A}\vct{x} - \vct{b}\|_2^2 
= \underbrace{\|\mtx{U}^\top(\mtx{A}\vct{x}-\vct{b})\|_2^2}_{\text{projected residual}} 
+ \underbrace{\|\mtx{A}\vct{x}_\ast - \vct{b}\|_2^2}_{\text{optimal residual}},
\]
where $\vct{x}_\ast$ is the minimizer.
The projected residual captures the part of the residual that LSQR actively reduces; the optimal residual is the irreducible error at the minimizer.

In exact arithmetic, LSQR is equivalent to applying conjugate gradients (CG) to the normal equations $\mtx{A}^\top\mtx{A}\vct{x}=\mtx{A}^\top\vct{b}$~\cite{paige1982lsqr}. Standard Krylov theory therefore yields the polynomial characterization
\begin{align}
\|\mtx{U}^\top\vct{r}_k\|_2^2
\le\;
\|\mtx{U}^\top\vct{r}_0\|_2^2
\min_{\substack{p_k\in\mathcal{P}_k\\p_k(0)=1}}
\max_{1\le i\le n} p_k(\sigma_i^2)^2,
\label{eq:weightedsum}
\end{align}
where $\mathcal{P}_k$ denotes polynomials of degree at most $k$.

This shows that LSQR convergence is governed by how well degree-$k$ polynomials approximate zeros over the spectrum of $\mtx{A}^\top\mtx{A}$,
\[
\sigma(\mtx{A}^\top\mtx{A}) = \{\sigma_1^2,\ldots,\sigma_n^2\}.
\]
Thus, convergence is governed by polynomial approximation on this set.

A classical bound based on Chebyshev polynomials gives
\begin{align}\label{eq:cvgbound_chebyshev}
\|\mtx{U}^\top\vct{r}_k\|_2
\le
2\left(\frac{\kappa(\mtx{A})-1}{\kappa(\mtx{A})+1}\right)^k
\|\mtx{U}^\top\vct{r}_0\|_2,
\end{align}
where $\kappa(\mtx{A})=\sigma_1/\sigma_n$~\cite[p.~187]{luenberger1973introduction}. However, this bound depends only on the extreme singular values and ignores interior spectral structure, and is therefore often pessimistic.

\subsubsection{Effect of spectral clustering}\label{sec:spectralclustering}

The polynomial characterization~\eqref{eq:weightedsum} shows that LSQR convergence is governed by the geometry of the spectral set $\sigma(\mtx{A}^\top\mtx{A})$. When the spectrum lies in a single interval, Chebyshev polynomials are optimal~\cite{fischer2011polynomial}. However, if the spectrum splits into disjoint intervals, substantially faster approximation can be achieved.

To formalize this, suppose that the squared singular values lie in the union of two disjoint intervals,
\[
\Omega = [a_1,a_2]\cup[a_3,a_4], \quad a_1<a_2<a_3<a_4.
\]
Define the associated \emph{asymptotic convergence factor}~\cite{eiermann1985study, fischer2011polynomial,schiefermayr2013estimatesasymptoticconvergence}
\[
\theta(\Omega)
= \lim_{k\to\infty}
\left(
\min_{\substack{p\in\mathcal{P}_k\\p(0)=1}}
\max_{x\in\Omega} |p(x)|
\right)^{1/k}.
\]
This quantity characterizes the optimal rate of polynomial decay on $\Omega$, and hence the asymptotic convergence of Krylov methods. In particular, \eqref{eq:weightedsum} implies
\[
\|\mtx{U}^\top\vct{r}_k\|_2
\lesssim
\theta(\Omega)^k
\|\mtx{U}^\top\vct{r}_0\|_2,
\]
so that $\theta(\Omega)$ determines the exponential rate of LSQR convergence.

A key feature is that $\theta(\Omega)$ decreases as the gap $a_3-a_2$ between the intervals increases~\cite[Theorem~3]{schiefermayr2013estimatesasymptoticconvergence}. Intuitively, a larger gap allows polynomials to decay more rapidly between the clusters, leading to faster convergence. Thus, spectral clustering can significantly accelerate Krylov methods even when the global condition number remains large.

In the LSQR setting, this corresponds to
\[
\sigma_1 \ge \cdots \ge \sigma_\ell \gg \sigma_{\ell+1} \ge \cdots \ge \sigma_n,
\]
so that $\sigma(\mtx{A}^\top\mtx{A})$ splits into two well-separated clusters. In particular, empirical results in~\cite{gergelits2014composite,carson2024towards} suggest that when the dominant singular values are tightly clustered, the effective convergence rate depends mainly on the conditioning of the trailing singular values rather than on the global condition number. This motivates the derivation of a bound for \cref{eq:weightedsum} in terms of the \emph{relative condition number} $\kappa_{\ell+1} = {\sigma_{\ell+1}}/{\sigma_n}$.

\begin{lemma}[Two-Interval Bound]\label{thm:cvgbound_disjoint}
    Let $\mtx{A} \in \F^{m \times n}$ have singular values $\sigma_1 \geq \cdots \geq \sigma_n$ such that $\sigma_1^2, \ldots, \sigma_\ell^2$ lie in an interval of width at most $w$, and $\sigma_{\ell+1}^2, \ldots, \sigma_n^2$ lie in a disjoint interval of width $w$. Then the $k$-th LSQR iterate satisfies
    \begin{align}\label{eq:convergence_quadbound}
        \|\mtx{U}^\top\vct{r}_k\|_2 
        \leq \left(\frac{\sqrt{C}-1}{\sqrt{C}+1}\right)^{\lfloor k/2 \rfloor}
        \|\mtx{U}^\top\vct{r}_0\|_2,
        \quad C := \kappa_{\ell+1}^2 + \frac{w}{\sigma_n^2}.
    \end{align}
\end{lemma}

\begin{proof}[Sketch of proof]
The proof constructs a polynomial adapted to the two-interval geometry by composing a quadratic mapping with a Chebyshev polynomial; see \cref{app:twointervalbound} for details.
\end{proof}

\noindent
\cref{thm:cvgbound_disjoint} formalizes the intuition that when the dominant singular values are tightly clustered, the convergence rate is governed primarily by the trailing spectrum through $\kappa_{\ell+1}$, rather than by the global condition number.

Overall, this perspective shows that LSQR convergence is controlled by the spectral distribution, and in particular that clustering the dominant singular values can significantly accelerate convergence even when the global condition number remains large. This observation motivates preconditioning strategies that forms such spectral structure.

\section{High-level Overview of \mainalg{}}\label{sec:highlevelalg}
\subsection{Low-rank spectral preconditioner}
The preceding analysis in \cref{sec:spectralclustering} suggests that an effective preconditioner does not need to reshape the entire spectrum; rather, it suffices to flatten the dominant singular values to a common level~\cite{coulaud2013deflation,gaul2013framework,gutknecht2012spectral}. Our goal is to construct a rank-$\l$ preconditioner $\mtx{P}_{\l,\mu} \in \F^{n \times n}$ such that the $\l$ largest singular values of $\mtx{A}_\mu\mtx{P}_{\l,\mu}^{-1}$ are approximately equal.

In the ideal setting where the rank-$\l$ truncated SVD $\lfloor{\mtx{A}}\rfloor_\l = \mtx{U}_\l \mtx{\Sigma}_\l \mtx{V}_\l^\top$ is available, the optimal rank-$\l$ preconditioner for the regularized linear least-squares problem~\cref{eq:augLLS} is given by
\begin{align*}
    \mtx{P}_{\l,\mu}^\star = \frac{1}{\sqrt{\sigma_{\l+1}^2+\mu^2}} \mtx{V}_\l (\mtx{\Sigma}_\l^2 +\mu^2 \mtx{I}_\l)^{1/2}\mtx{V}_\l^\top + (\mtx{I}-\mtx{V}_\l\mtx{V}_\l^\top).
\end{align*}
This collapses the first $\l$ regularized dominant singular values to the level $\sqrt{\sigma_{\l+1}^2+\mu^2}$, achieving the best possible condition number reduction for a rank-$\ell$ modification:
\begin{align*}
    \kappa_2(\mtx{A}_\mu) = \sqrt{\frac{\sigma_1^2 + \mu^2}{\sigma_n^2+\mu^2}}\quad \rightarrow \quad\kappa_2(\mtx{A}_\mu \mtx{P}_\star^{-1}) = \sqrt{\frac{\sigma_{\l+1}^2 + \mu^2}{\sigma_n^2 + \mu^2}}.
\end{align*}

Since computing the truncated SVD is infeasible for large-scale problems, we approximate $\mtx{P}_{\l,\mu}^\star$ using a randomized \emph{CUR approximation}. CUR approximation is particularly well suited to our adaptive setting because it supports efficient incremental rank updates~\cite{pritchard2025fast} and preserves structural properties such as sparsity by selecting actual rows and columns of $\mtx{A}$ (see \cref{sec:whycur} for details).

\begin{definition}[Rank-$\l$ CUR approximation~\cite{zamarashkin2018existence,osinsky2025close}]\label{def:cur}
    A rank-$\l$ CUR approximation $\widehat{\mtx{A}}_\ell$ of matrix $\mtx{A}$ with \emph{target rank} $\l$ is given by
    \begin{equation}\label{eq:cur}
        \mtx{A} \approx \widehat{\mtx{A}}_\l = \mtx{C}\mtx{U}\mtx{R},
    \end{equation}
    where $\mtx{C}=\mtx{A}(:,\vct{J})\in\F^{m \times \l}$ and $\mtx{R}=\mtx{A}(\vct{I},:)\in\F^{\l \times n}$ 
    consist of $\l$ selected columns and rows of $\mtx{A}$ indexed by the sets $\vct{J} \subseteq \{1, \ldots, n\}$ and $\vct{I} \subseteq \{1, \ldots, m\}$, respectively. The \textit{core matrix} $\mtx{U} = \mtx{A}(\vct{I},\vct{J})^\dagger\in \F^{\l \times \l}$ is chosen to ensure a small approximation error.
\end{definition}

\noindent
CUR provides powerful approximation: for any matrix and rank $\ell$, there exists a rank-$\ell$ CUR whose error is within a factor $(\ell+1)$ of the optimal, truncated SVD in the Frobenius norm, indicating the sacrifice in accuracy by choosing CUR is minimal, and practical algorithms are now available~\cite{osinsky2025close}.

By leveraging the low-rank structure of the CUR decomposition, we can efficiently compute the rank-$\l$ truncated SVD of the approximation, $\widehat{\mtx{A}}_\l = \widehat{\mtx{U}}\Lhat\Vhat^\top$. Using these spectral components, we define the rank-$\l$ CUR preconditioner:
\begin{align}\label{eq:preconditioner2}
    \mtx{P}_{\l,\mu} = \frac{1}{\svhat}\Vhat(\Lhat^2 + \mu^2 \mtx{I})^{1/2}\Vhat^\top + (\mtx{I} - \Vhat \Vhat^\top),
\end{align}
where $\l$ is a \textit{target rank} and $\svhat > 0$ is a prescribed \textit{target level}. Note that when $\Vhat = \mtx{V}_\l$ and $\svhat = \sqrt{\svhat_{\l+1}^2 + \mu^2}$, $\mtx{P}_{\l,\mu}$ recovers the optimal preconditioner $\mtx{P}_{\l,\mu}^\star$.


\subsection{Adaptive strategy}\label{sec:adaptivity}
The effectiveness of a low-rank spectral preconditioner depends critically on the choice of rank $\ell$. 
If $\ell$ is too small, the dominant singular components remain insufficiently flattened, leading to slow convergence. 
If $\ell$ is too large, unnecessary spectral modification may destroy useful separation between dominant and trailing singular values and incur avoidable computational cost; see \cref{sec:numericalstudies} for numerical evidence. 
Since the appropriate rank is problem-dependent and typically unknown a priori, fixing $\ell$ in advance is inherently unreliable.

To address this, we introduce an adaptive framework with two components: \emph{rank adaptivity}, which determines how much spectral correction to apply, and \emph{schedule adaptivity}, which determines when to apply it. 
Together, these enable the preconditioner to be refined progressively during the solve rather than fixed upfront, allowing the method to start with a modest-rank preconditioner and refine it only when additional spectral correction becomes beneficial. This approach amortizes setup costs over the iterations while maintaining rapid convergence.


\subsubsection{Rank adaptivity: what to precondition}

The first form of adaptivity concerns the choice of rank $\l$. In many large-scale problems, the numerical rank of $\mtx{A}$ is not known a priori and may vary across applications, making a fixed choice of $\l$ unreliable. We therefore develop a rank-adaptive approach that incrementally refines the spectral correction without prior knowledge of the target rank.

Specifically, we construct a sequence of CUR approximations
\begin{align*}
\{\widehat{\mtx{A
}}_{\l}\}_{\l = b, 2b, 3b, \ldots}
\end{align*}
where the target rank increases in blocks of size $b$. The rank is grown until a prescribed approximation accuracy is reached.

For computational efficiency, the preconditioner is not updated at every rank increment. Instead, updates are triggered only when the improvement in approximation accuracy indicates that additional spectral flattening is likely to yield meaningful convergence gains. This results in a subsequence
$\{\widehat{\mtx{A}}_{\l_i}\}_{i=1}^p$
used to construct a sequence of CUR preconditioners
$$\{\mtx{P}_{\l_i,\mu}\}_{i=1}^p,$$
where $\l_i$ denotes the target rank at the $i$-th update and $p$ is the total number of preconditioner updates.


This strategy avoids both under-preconditioning, where too few dominant components are corrected, and over-preconditioning, where excessive rank growth incurs unnecessary cost and may even impair convergence.
Rank adaptivity thus determines how many leading spectral components should be modified by our preconditioners.

To form CUR approximations of increasing ranks, we employ the randomized \textit{IterativeCUR} method of Pritchard \textit{et al.}~\cite{pritchard2025fast} (see \cref{alg:cur_icur}), which progressively constructs CUR approximations from a single sketch $\mtx{SA}$ computed once at initialization, where $\mtx{S} \in \mathbb{F}^{b \times m}$ is a random embedding, such as Gaussian matrices (see \cref{app:cur_details}).
The approximation is refined incrementally by selecting additional rows and columns based on the sketched residual $\mtx{SA} - \mtx{SCUR}$, reusing the same sketch throughout. This enables rank refinement without additional sketching cost, ensuring low overhead at each update and making the approach well-suited to the adaptive framework. Importantly, the sketch dimension $b$ is independent of the final target rank.

\subsubsection{Schedule adaptivity: when to precondition}

The second form of adaptivity concerns when preconditioner updates should occur. In contrast to static approaches that construct a complete preconditioner before solving, we interleave preconditioner updates with LSQR iterations, amortizing setup cost across the solve.

This strategy is supported by the spectral behavior of Krylov methods. LSQR reduces error components associated with large singular values in early iterations, while components associated with smaller singular values become relevant only later~\cite{paige1982lsqr,bjorck1996numericalmethods}. Accordingly, only a small number of dominant spectral components need to be corrected initially, whereas additional spectral refinement becomes relevant as convergence progresses.

%

Starting from an initial guess $\vct{x}_0$, with phase index $i=1$ and initial rank $\l_1=b$, we proceed as follows: (i) construct a rank-$\l_i$ CUR approximation $\widehat{\mtx{A}}_{\l_i}$; (ii) form the corresponding preconditioner $\mtx{P}_{\l_i,\mu}$; (iii) run preconditioned LSQR (PLSQR) while monitoring convergence; and (iv) when convergence slows, set $\vct{x}_i$ to the current iterate, increment $i \gets i+1$, update the CUR approximation until the new target rank $\l_i$, and repeat (ii)–(iv). 

The algorithm terminates after step (iii) once the CUR approximation meets the prescribed accuracy. Each LSQR phase is warm-started from the previous iterate:
\begin{align*}
    \vct{x}_i = \vct{x}_{i-1} + \mathrm{PLSQR}(\mtx{A}_\mu, \mtx{P}_{\l_i,\mu}, \vct{b}_{\mathrm{aug}} - \mtx{A}_\mu \vct{x}_{i-1}).
\end{align*}
Fortunately, due to the residual-minimizing property of LSQR, the residual remains monotonically nonincreasing throughout the LSQR iterations even with the warm-restarts. Additionally, this process fortuitously performs LSQR iterative refinement, which has shown to be critical for achieving stability in preconditioned LSQR~\cite{epperly2024fastforwardstablerandomized,epperly2026fast}.

A modest-rank preconditioner is often sufficient to accelerate early iterations, but its effectiveness diminishes once the corrected singular components have been fully exploited. Schedule adaptivity detects this slowdown and introduces additional spectral correction only when needed.

Together, rank adaptivity and schedule adaptivity allow the preconditioner to evolve alongside the Krylov iterations, ensuring that computational effort scales naturally with the requested accuracy.



\subsection{Advantages of the CUR approximation}\label{sec:whycur}
While other randomized low-rank approximation methods, such as randomized SVD \cite{halko2011finding} or generalized Nyström methods \cite{nakatsukasa2020fast, tropp2017practical} are powerful, the CUR approximation in \eqref{eq:cur} offers unique advantages that align specifically with the requirements of our adaptive and efficient framework.

Most importantly, CUR is naturally suited for incremental rank growth. Increasing the approximation rank simply involves augmenting the existing row and column sets, thereby reusing previous computations rather than restarting from scratch. 
Furthermore, by selecting actual rows and columns of $\mtx{A}$ to form the bases $\mtx{C}$ and $\mtx{R}$, CUR inherits the intrinsic structural properties of the original matrix, such as sparsity and non-negativity. This preservation of structure often leads to more physically meaningful approximations and improved memory efficiency, as the decomposition can be represented by storing indices rather than dense full matrices.




\subsection{Overview of the main algorithm}

We now summarize the high-level structure of the proposed method, \mainalg{}.
\cref{fig:aplicur-flow} illustrates the overall workflow. 
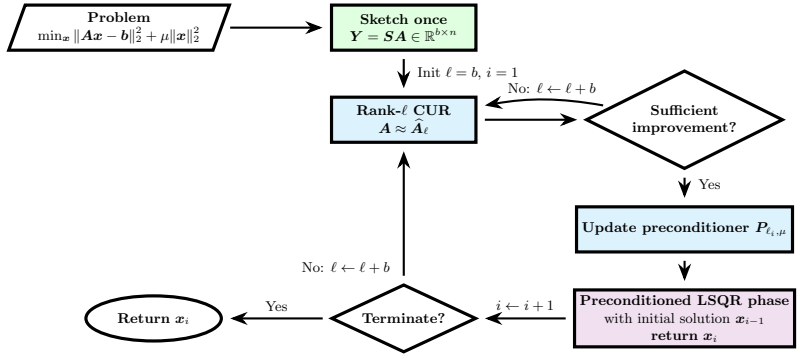
\begin{figure}
    \centering
    \begin{tikzpicture}[scale=0.6, transform shape, node distance=12mm and 24mm, every node/.style={font=\small}]
        \node[io] (input) {\textbf{Problem}\\$\min_{\vct{x}} \|\mtx{Ax}-\vct{b}\|_2^2+ \mu\|\vct{x}\|_2^2$};
        \node[process, right=of input, fill=green!12] (sketch) {\textbf{Sketch once}\\$\mtx{Y} = \mtx{SA} \in \mathbb{R}^{b \times n}$};
        \draw[arrow] (input) -- (sketch);
        \node[process, below=10mm of sketch, fill=cyan!12] (cur) {\textbf{Rank-$\l$ CUR}\\$\mtx{A} \approx \mtx{\widehat{A}}_\l$};
        \draw[arrow] (sketch) -- node[midway, right, xshift=5pt] {Init $\l = b$, $i=1$}(cur); 
        \node[decision, right=of cur] (decide) {\textbf{Sufficient}\\\textbf{improvement?}};
        \draw[arrow] (cur) -- (decide);
        \node[process, below=8mm of decide, fill=cyan!12] (precondition) {\textbf{Update preconditioner }$\mtx{P}_{\l_i,\mu}$};
        \draw[arrow] (decide) -- node[pos=0.4, right, xshift=5pt] {Yes} (precondition);
        \draw[arrow] (decide) to[bend right=10] node[midway,above, xshift=5pt] {No: $\l \gets \l+ b$} (cur);
        \node[process, below=8mm of precondition, fill=violet!12] (lsqr) {\textbf{Preconditioned LSQR phase}\\with initial solution $\vct{x}_{i-1}$ \\\textbf{return} $\vct{x}_i$};
        \draw[arrow] (precondition) -- (lsqr);
        \node[decision] (termination) at ($(cur |- lsqr)$) {\textbf{Terminate?}};
        \draw[arrow] (lsqr) -- node[midway,above, xshift=0pt, yshift=1pt]{$i \gets i+1$} (termination);
        \node[startstop, left=of termination] (stop) {\textbf{Return} $\vct{x}_i$};
        \draw[arrow] (termination) -- node[midway, above, yshift=1pt] {Yes} (stop);
        \draw[arrow] (termination) to[bend left=00] node[pos=0.1, left, xshift=-5pt] {No: $\l \gets \l+ b$} (cur);
    \end{tikzpicture}
    \caption{Workflow of \mainalg. A single sketch $\mtx{Y}=\mtx{SA}$ is computed once and reused to select the next row and column indices in the CUR updates. The algorithm alternates between (i) CUR updates and conditional preconditioner construction, and (ii) warm-started preconditioned LSQR phases. No additional sketching is required as the rank grows.}
    \label{fig:aplicur-flow}
\end{figure}
After forming a sketch $\mtx{SA}$ once at the initialization, \mainalg{} alternates between two phases:
\begin{itemize}
    \item \textbf{Preconditioning phase.} The CUR approximation is incrementally refined by increasing the target rank in blocks of size $b$, improving the approximation of $\mtx{A}$ without any resketching. When the update criteria are satisfied, the preconditioner $\mtx{P}_{\l,\mu}$ is updated accordingly.
    \item \textbf{Solving phase.} LSQR is applied to \eqref{eq:augLLS} using the current preconditioner, warm-started from the previous iterate. The number of LSQR iterations per phase is determined by dynamic stopping criteria that monitor changes in the convergence rate.
\end{itemize}
These phases are interleaved until the CUR approximation reaches the prescribed accuracy, at which point the algorithm terminates.

We present a high-level pseudocode description of the proposed method in \cref{alg:aplicur_highlevel}. This highlights the interaction between rank adaptivity and schedule adaptivity. A detailed implementation, including stopping criteria and efficient updates, is given in \cref{alg:aplls}.
\begin{algorithm}
\caption{APLICUR (generic algorithmic framework)}
\label{alg:aplicur_highlevel}
\begin{algorithmic}[1]
\Require $\mtx{A} \in \mathbb{R}^{m \times n}$, $\vct{b} \in \mathbb{R}^m$, regularization $\mu \ge 0$, block size $b$
\Ensure Approximate solution $\mtx{x}$

\State Compute sketch $\mtx{Y} = \mtx{SA}$
\State Initialise $\vct{x}_0 = \vct{0}$, $\l_1 = 0$, $i = 1$

\Repeat
    \Repeat
        \State $\l_i \gets \l_i + b$
        \State Construct CUR approximation $\widehat{\mtx{A}}_{\l_i}$ using $\mtx{Y}$
    \Until{CUR approximation has sufficiently improved}
    \State Form preconditioner $\mtx{P}_{\l_i,\mu}$ from $\widehat{\mtx{A}}_{\l_i}$ via~\cref{eq:preconditioner2}
    \State Start running preconditioned LSQR with $\mtx{P}_{\l_i,\mu}$, warm-started at $\vct{x}_{i-1}$
    \If{convergence slows}
        \State Let $\vct{x}_i$ be the current LSQR 
        iterate
        \State $i \gets i + 1$ and $\l_i \gets \l_{i-1}$
    \EndIf
\Until{desired accuracy achieved}

\State \Return $\vct{x}_i$
\end{algorithmic}
\end{algorithm}

\section{Theoretical Analysis}\label{sec:theory}

In this section, we establish convergence guarantees for the adaptively preconditioned LSQR iterations within the APLICUR framework (see \cref{alg:aplicur_highlevel}). The method applies a sequence of preconditioners $\{\mtx{P}_{\l_i,\mu}\}_{i=1}^p$ over $p$ successive LSQR phases to the augmented linear least-squares problem in \eqref{eq:augLLS}, progressively improving its conditioning.

Our analysis characterizes how the condition number and spectral structure evolve across phases, and shows how this adaptive strategy accelerates convergence.

\subsection{Singular value bounds for the preconditioned system}
We begin by quantify the spectral effects induced by the CUR-based preconditioner $\mtx{P}_{\l,\mu}$, through bounds on the singular values of $\mtx{A}\mtx{P}_{\l,\mu}^{-1}$.

\begin{theorem}[Condition Number Bound]\label{thm:conditionnumberbound}
    Let $\mtx{\widehat A}_\l$ be a rank-$\ell$ CUR approximation of $\mtx A$ with residual $\mtx E = \mtx A - \mtx {\widehat A}_\l$, and define the augmented matrix
    $$\mtx{A}_\mu = \begin{bmatrix} \mtx{A} \\ \mu \mtx{I} \end{bmatrix}, \qquad \mu > 0.$$
    Let $\mtx{P}_{\l,\mu}$ be the preconditioner in \eqref{eq:preconditioner2} constructed from $\mtx{\widehat A}_\l$ with target level $\svhat$ satisfying $\svhat^2 \le \svhat_\l^2 + \mu^2$.
    Assuming $\|\mtx{E}\|_2 < \mu$, we have
    \[\sigma_{\max}(\mtx A_\mu \mtx{P}_{\l,\mu}^{-1}) \le \sqrt{\widehat\sigma^2 + \mu^2} + \|\mtx E\|_2,
    \qquad
    \sigma_{\min}(\mtx A_\mu \mtx{P}_{\l,\mu}^{-1}) \ge \mu - \|\mtx E\|_2, \]
    and therefore
    \[ \kappa(\mtx A_\mu \mtx{P}_{\l,\mu}^{-1})
    \le \frac{\sqrt{\widehat\sigma^2 + \mu^2} + \|\mtx E\|_2} {\mu - \|\mtx E\|_2}. \]
\end{theorem}

\begin{proof}
    Let $\mtx {\widehat A}_\l = \mtx {\widehat U \widehat\Sigma \widehat V}^\top$ be the
    rank-$\ell$ truncated SVD, and define
    \[ \mtx {\widehat A}_{\l,\mu} = \begin{bmatrix} \mtx {\widehat A}_\l \\ \mu \mtx I \end{bmatrix},
    \qquad
    \Vhat_\mathrm{full} = \begin{bmatrix}  \Vhat & \Vhat_\perp \end{bmatrix}.\]
    By the construction of $\mtx{P}_{\l,\mu}$,
    \[ \mtx{P}_{\l,\mu}^{-\top}\mtx {\widehat A}_{\l,\mu}^\top \mtx {\widehat A}_{\l,\mu} \mtx{P}_{\l,\mu}^{-1}
    = \mtx{P}_{\l,\mu}^{-\top}(\mtx{\widehat{A}}_\l^\top \mtx{\widehat{A}}_\l+\mu^2 \mtx{I_n} ) \mtx{P}_{\l,\mu}^{-1}
    = \Vhat_\mathrm{full} \begin{bmatrix} (\svhat^2+\mu^2)\mtx{I}_\l & \\ & \mu^2 \mtx{I}_{n-\l} \end{bmatrix} \Vhat_\mathrm{full}^\top, \]
    so that 
    \[ \sigma_{\max}(\mtx {\widehat A}_{\l,\mu} \mtx{P}_{\l,\mu}^{-1}) = \sqrt{\widehat\sigma^2 + \mu^2},
    \qquad
    \sigma_{\min}(\mtx {\widehat A}_{\l,\mu} \mtx{P}_{\l,\mu}^{-1}) = \mu. \]
    
    Moreover,
    \[ \mtx A_\mu \mtx{P}_{\l,\mu}^{-1} = \mtx {\widehat A}_{\l,\mu} \mtx{P}_{\l,\mu}^{-1} + \begin{bmatrix} \mtx E \\ \mtx 0 \end{bmatrix} \mtx{P}_{\l,\mu}^{-1}
    =: \mtx {\widehat A}_{\l,\mu} \mtx{P}_{\l,\mu}^{-1} + \mtx\Delta, \]
    and since $\svhat^2 \le \svhat_\l^2 + \mu^2$, we have $\|\mtx{P}_{\l,\mu}^{-1}\|_2 \le 1$, and therefore $\|\mtx{\Delta}\|_2 \le \|\mtx{E}\|_2$.

    Weyl’s inequality for singular values then yields the stated bounds. 
    See \cref{proof:conditionnumberbound} for details.
\end{proof}
\noindent
Note that the condition number bound in \cref{thm:conditionnumberbound} depends only on the target level $\svhat$, the CUR approximation error $\|\mtx{E}\|_2$, and the regularization parameter $\mu$. Importantly, the sketch dimension does not appear in the bound.

\subsection{Convergence bound for adaptive CUR preconditioning}

We now analyze the convergence of \mainalg{} (\cref{alg:aplicur_highlevel}), which applies LSQR with a sequence of CUR-based preconditioners. We begin by considering a single preconditioned LSQR phase, and then extend the result to the full multi-phase adaptive procedure.

Let $\vct{r}_k := \mtx{A}_\mu \vct{x}_k - \vct{b}_{\text{aug}}$ denote the residual after the $k$-th LSQR iteration, and let \(\mtx{U}_\mu\) be the matrix of left singular vectors of \(\mtx{A}_\mu\). Applying the classical Chebyshev bound from \cref{eq:cvgbound_chebyshev} to the preconditioned system, together with \cref{thm:conditionnumberbound}, yields
\begin{align*}
    \frac{\|\mtx{U}_\mu^\top\vct{r}_k\|_2}{\|\mtx{U}_\mu^\top\vct{r}_0\|_2} \leq 2 \left(\frac{\kappa(\mtx A_\mu \mtx{P}_{\l,\mu}^{-1})-1}{\kappa(\mtx A_\mu \mtx{P}_{\l,\mu}^{-1})+1}\right)^k \leq 2 
    \left(
    1 - \frac{2\left(\mu - \|\mtx{E}\|_2\right)}
    {\mu + \sqrt{\svhat^2 + \mu^2}}
    \right)^{k}, \quad k \geq 0.
\end{align*}




Our main algorithm \mainalg{} (\cref{alg:aplicur_highlevel}) proceeds through multiple LSQR phases with progressively refined preconditioners. The following lemma captures the cumulative effect of these phases.
\begin{lemma}\label{thm:cvgbound_aplicur_naive}
    Consider \mainalg{} with $p$ LSQR phases. In phase $i$, we perform $k^{(i)}$ LSQR iterations using a preconditioner $\mtx{P}_{\ell_i,\mu}$ constructed from a rank-$\ell_i$ CUR approximation $\mtx{\widehat A}_{\ell_i}$ of $\mtx A$, with residual $\mtx E^{(i)} = \mtx A - \mtx{\widehat A}_{\ell_i}$, and target level $\widehat{\sigma}^{(i)}$ satisfying $(\widehat{\sigma}^{(i)})^2 \le \widehat{\sigma}_{\ell_i}^2 + \mu^2$. Assume that $\|\mtx E^{(i)}\|_2 < \mu$ for all $i=1,\ldots,p$.

    Then, after a total of $\sum_{i=1}^p k^{(i)}$ LSQR iterations, the residual satisfies
    \[
    \left\|\mtx{U}_\mu^\top\vct{r}_k\right\|_2
    \le
    2^p
    \prod_{i=1}^p
    \left(
    1 - \frac{2\bigl(\mu - \|\mtx{E}^{(i)}\|_2\bigr)}
    {\mu + \sqrt{(\widehat{\sigma}^{(i)})^2 + \mu^2}}
    \right)^{k^{(i)}}
    \|\mtx{U}_\mu^\top \vct{r}_0\|_2.
    \]
\end{lemma}
\begin{proof}
    It follows directly from \cref{eq:weightedsum,thm:conditionnumberbound}.
\end{proof}
\noindent
The bound shows that each phase contributes a contraction factor determined by the current CUR accuracy and target level. As the rank increases, the CUR accuracy usually improves (often significantly, i.e., \(\|\mtx{E}^{(i+1)}\|_2 \ll \|\mtx{E}^{(i)}\|_2\)), so the contraction improves monotonically.

\begin{remark}
    In practice, convergence within each phase is often rapid in the first few LSQR iterations, even when the CUR approximation error is relatively large. This reflects the convergence behavior of LSQR, where early iterations primarily reduce components associated with the largest singular values, which are already captured by the current low-rank approximation. Consequently, only a small number of additional spectral correction is enough in each phase to maintain fast convergence.
\end{remark}


Using standard CUR approximation guarantees (see \cref{thm:curosinsky}), the rank-\(\ell\) CUR approximation error satisfies
$$\|\mtx{E}\|_2 = \|\mtx{A} - \mtx{\widehat{A}}_\ell\|_2 \le C(\ell)\,\sigma_{\ell+1}(\mtx{A}),$$
where $C(\l)$ depends polynomially on $\l$ and is independent of the sketch size.

Combining this with \cref{thm:cvgbound_aplicur_naive} yields the following convergence bound for \mainalg{}, in terms of the target levels \(\{\widehat{\sigma}^{(i)}\}_{i=1}^p\), target ranks \(\{\l_i\}_{i=1}^p\), and the spectral decay of \(\mtx{A}\), as captured by the largest trailing singular values \(\{\sigma_{\l_i+1}\}_{i=1}^p\).

%

\begin{corollary}[Convergence bound for \cref{alg:aplicur_highlevel}]
Under the assumptions of Lemma~\ref{thm:cvgbound_aplicur_naive}, the residual of \mainalg{} after $k$ total LSQR iterations satisfies
\[
\left\|\mtx{U}_\mu^\top\vct{r}_k\right\|_2
\le
2^p
\prod_{i=1}^p
\left(
1 - \frac{2\bigl(\mu - C(\ell_i)\sigma_{\ell_i+1}(\mtx{A})\bigr)}
{\mu + \sqrt{(\widehat{\sigma}^{(i)})^2 + \mu^2}}
\right)^{k^{(i)}}
\|\mtx{U}_\mu^\top \vct{r}_0\|_2.
\]
\end{corollary}

\section{Adaptively Preconditioned LSQR via iterative CUR (\mainalg{}): implementation}\label{sec:implementation}
We now provide a detailed implementation of \cref{alg:aplicur_highlevel}, including the construction of the CUR approximation, preconditioner updates, and adaptive criteria.

\subsection{Single sketch rank-adaptive CUR approximation}
We briefly describe how to construct a sequence of CUR approximations with increasing rank using IterativeCUR~\cite{pritchard2025fast}, using only a single sketch. While numerous CUR variants exists---differing in sketching methods, index selection strategies, core matrix construction, and rank choice---we focus here on the specific variant used in our implementation. A comprehensive overview of alternative approaches is provided in \cref{app:cur_details}.

Let $\mtx{S}\in \F^{s \times m}$ be a random embedding of sketch dimension $s = 1.1b$ for a prescribed block size $b$. Then, we form the sketch $\mtx{Y} = \mtx{SA}$. In our implementation, we employ a sparse sign embedding (\cref{def:sparsesign}), which allows the sketch to be formed in $\mathcal{O}(\xi \cdot \text{nnz}(\mtx{A}))$ operations, where $\xi$ is the sparsity parameter of the embedding. Crucially, this sketch is computed once and reused throughout the algorithm.

The sparse sign embedding is particularly efficient when $\mtx{A}$ is sparse, as the cost of applying the sketch scales linearly with $\text{nnz}(\mtx{A})$, the number of nonzero entries in $\mtx{A}$~\cite{clarkson2017low}. Moreover, it ensures the sketched sparse matrix to remain sparse. 
\yn{YN: My understanding is that in IterativeCUR, there is almost no point in using a sparse sketch, as the size is $O(1)$. Relatedly, I think this whole subsection should be removed for the paper, being a summary of \cite{pritchard2025fast}.}
\begin{definition}[Sparse sign embedding~\cite{nelson2013osnap, woodruff2014sketching, park2025accuracy}]\label{def:sparsesign}
    A \textit{sparse sign embedding} is a sparse random matrix of the form:
    \begin{align*}
        \mtx{S} = \frac{1}{\sqrt{\xi}} \begin{bmatrix}
            \vct{c}_1 & \cdots & \vct{c}_m
        \end{bmatrix}\in \mathbb{R}^{s \times m},
    \end{align*}
    where each column $\vct{c}_i$ is drawn independently and contains $\xi$ independent Rademacher random variables placed at uniformly random coordinates. The parameter $\xi$ is called the \textit{sparsity parameter}.
\end{definition}

For incremental rank growth, we follow the procedure of IterativeCUR.
For a rank-$\l$ CUR approximation $\mtx{\widehat{A}}_\l$ with chosen row and column index sets $\vct{I}$ and $\vct{J}$,
define the sketched residual $$\mtx{E}^{\mathrm{row}} = \mtx{S}(\mtx{A} - \mtx{\widehat{A}}_\l) = \mtx{Y} - \mtx{Y}(:,\vct{J})\mtx{A}(\vct{I},\vct{J})^\dagger\mtx{A}(\vct{I},:).$$The next $b$ column indices, denoted $\vct{J_+}$, are selected by applying pivoted factorization to ${\mtx{E}^\mathrm{row}}^\top$. We then determine the next $b$ row indices, denoted $\vct{J_+}$, by applying a pivoted factorization to the column residual $\mtx{E}^\mathrm{col}=\mtx{A}(:,\vct{J_+}) - \mtx{\widehat{A}}_\l(:,\vct{J_+})$. The index sets are then augmented as $\vct{J}\gets [\vct{J}, \vct{J_+}]$ and $\vct{I} \gets [\vct{I}, \vct{I_+}]$, yielding a rank-$(\l+b)$ CUR approximation accordingly. In IterativeCUR, the effect of the choice of index selection method on accuracy is negligible. Hence, in our implementation we select column and row indices via LUPP factorization. Detailed implementation is provided in \cref{alg:cur_icur,alg:cur_augment}.

A key advantage of this strategy is that no additional sketches are required during rank augmentation. The sketch dimension $s$ remains fixed and is independent of the matrix size or numerical rank. Consequently, IterativeCUR realizes the rank-adaptivity without incurring additional sketching cost.

\subsection{CUR preconditioner}\label{sec:qrsvd}
Given a rank-$\l$ CUR approximation $\mtx{\widehat{A}}_\l$, we must construct the rank-$\l$ spectral preconditioner $\mtx{P}_{\l,\mu}$ defined in \eqref{eq:preconditioner2}, which requires the truncated SVD $\mtx{\widehat{A}}_\l = \mtx{\widehat{U}\widehat{\Sigma}\widehat{V}}^\top$.

Since directly computing SVD on the CUR approximation would be prohibitive for large-scale problems, we exploit the low-rank structure. Let $\mtx{C}=\mtx{Q_C T_C}$ and $\mtx{R^\top}=\mtx{Q_R T_R}$ be thin QR factorizations (in practice implemented via sketched Cholesky QR for efficiency~\cite{balabanov2022randomized,fukaya2014}). Then, we have 
\begin{align}
    \mtx{CUR}=\mtx{Q_c}\mtx{M}\mtx{Q_R^\top},
\end{align} 
where the middle matrix $\mtx{M}:= \mtx{T_C U T_R^\top}\in \F^{\l \times \l}$ serves as a compact spectral surrogate of the CUR approximation, capturing its dominant spectral properties. We compute the SVD $$\mtx{M}=\mtx{\widehat{U}_M \Lhat_M \widehat{V}_M^\top},$$ and obtain the singular values $\Lhat = \Lhat_M$ and right singular vectors $\Vhat = \mtx{Q_R \widehat{V}_M}$. This requires an SVD only of size $\l \times \l$, whose cost is $\mathcal{O}(\l^3)$.

The preconditioner is applied implicitly via matrix-vector products; no dense $n \times n$ matrix is formed. Details of the implementation is provided in \cref{alg:preconditioner}.

\begin{algorithm}[H]
    \caption{\texttt{CURPreconditioner}}
    \label{alg:preconditioner}
    \begin{algorithmic}[1]
        \Require{Rank-$\l$ CUR approximation $\mtx{C, U, R}$, regularization parameter $\mu$.}
        \Ensure Preconditioner $\mtx{P}_{\l,\mu}$

        \Statex
        \State \textbf{Part 1: Factorize $\mtx{C}$ and $\mtx{R}$}
        \State $\mtx{T_C} \gets \texttt{chol}(\mtx{C}^\top \mtx{C})$\Comment{or sketched Cholesky QR for stability}
        \State $[\mtx{Q_R}, \mtx{T_R}] \gets \texttt{qr}(\mtx{R^\top})$
        \Comment{or sketched Cholesky QR for efficiency}
        
        \Statex
        \State \textbf{Part 2: Build spectral preconditioner}
        \State $[\sim, \Lhat, \Vhat_{M}] \gets \texttt{svd}(\mtx{T_C}\mtx{U}\mtx{T_R^\top})$
        \State $\Vhat \gets \mtx{Q}_{\mtx{R}} \Vhat_{\mtx{M}}$ \Comment{use implicit form $\mtx{R}^\top \mtx{T}_{\mtx{R}}^{-1} \Vhat_{\mtx{M}}$ in sparse case}
        \State $\Lhat \gets \texttt{sqrt}(\texttt{diag}(\svhat_1^2 + \mu^2, \ldots, \svhat_\l^2 + \mu^2))$
        \State $\svhat \gets \texttt{sqrt}(\svhat_{\l}^2 + \mu^2)$ \Comment{target level}
        \State $\mtx{P}_{\l,\mu} \gets \svhat^{-1} \Vhat\Lhat\Vhat^\top + (\mtx{I}_n - \Vhat\Vhat^\top$)
        \Comment{apply via \texttt{matvec}}
    \end{algorithmic}
\end{algorithm}

\subsubsection{SVD-free variant}\label{sec:svdfreepreconditioner}
Although the strategy described above is reducing the SVD cost of an $m \times n$ matrix to that of an $\l \times \l$ matrix, the smaller SVD can still dominate the preconditioner construction cost when $\l$ is large. We therefore introduce an SVD-free variant of the preconditioner in the unregularized case (i.e., $\mu = 0$ in \eqref{eq:regularizedLLS}).

With the QR factorizations $\mtx{C}=\mtx{Q_C}\mtx{T_C}$ and $\mtx{R^\top} = \mtx{Q_R}\mtx{T_R}$, recall that we can form the small core matrix $\mtx{M} = \mtx{T_C} \mtx{A}(\vct{I},\vct{J})^\dagger \mtx{T_R^\top} \in \F^{\l \times \l}$.
Rather than computing an SVD of $\mtx{M}$, we can define the rank-$\l$ SVD-free preconditioner directly in terms of these matrices:
\begin{align}\label{eq:preconditioner_svdfree}
    \mtx{\tilde{P}}_\l &= \svhat^{-1}\mtx{Q_R} \mtx{M}\mtx{Q_R}^\top + (\mtx{I_n}-\mtx{Q_R Q_R^\top}),\notag\\
    \mtx{\tilde{P}}_\l^{-1} &= \svhat\mtx{Q_R} \mtx{M}^{-1}\mtx{Q_R}^\top + (\mtx{I_n}-\mtx{Q_R Q_R^\top}).
\end{align}
\noindent
When $\mtx{M}$ is nonsingular, its inverse can be explicitly written as
\begin{align*}
    \mtx{M}^{-1} = \mtx{T_R}^{-\top}\mtx{A}(\vct{I},\vct{J}) \mtx{T_C}^{-1},
\end{align*}
so applying $\mtx{M}^{-1}$ requires only two $\ell\times \ell$ triangular solves 
and a multiplication by the small intersection matrix $\mtx{A}(\vct{I}, \vct{J})$. The nonsingularity of $\mtx{M}$ is guaranteed provided that the target rank $\l$ does not exceed the numerical rank of $\mtx{A}$, and sensible sets of findices are found for C and R---a condition  naturally enforced by the adaptive rank selection strategy in our implementation. Details of the implementation is provided in \cref{alg:svdfreepreconditioner}.


Although the SVD-free variant differs algebraically from the SVD-based preconditioner in \cref{eq:preconditioner2}, the two are equivalent up to an orthogonal transformation, as established in the next lemma.
\begin{lemma}\label{lemma:svdfreepreconditioner}
    Let $\mtx{P}$ denote the SVD-based preconditioner defined in \cref{eq:preconditioner2}, and let $\mtx{\widetilde{P}}$ denote the SVD-free preconditioner defined in \cref{eq:preconditioner_svdfree}. Then, we have
    \begin{align}
        \mtx{\tilde{P}}_\l^{-1}  = \mtx{P}_\l^{-1} \mtx{W},
    \end{align}
    where $\mtx{W}$ is an orthogonal matrix.

    Consequently, the preconditioned matrices share the same singular values:
    \begin{align*}
        \sigma_i (\mtx{A}\mtx{\tilde{P}}_\l^{-1}) = \sigma_i (\mtx{AP}_\l^{-1})\quad \text{for all $i$}.
    \end{align*}
\end{lemma}
\begin{proof}
    Using the decompositions from \cref{sec:qrsvd}, write
    \begin{align*}
        \mtx{M}=\mtx{\widehat{U}_M \Lhat_M \widehat{V}_M^\top}, \quad \Vhat = \mtx{Q_R \widehat{V}_M}.
    \end{align*}
    Then the inverse of the SVD-free preconditioner \cref{eq:preconditioner_svdfree} becomes
    \begin{align*}
        \mtx{\tilde{P}}_\l^{-1} 
        &= \svhat \Vhat \mtx{\Lhat}^{-1} (\mtx{\widehat{U}_M^\top \Vhat_M}) \Vhat^\top + (\mtx{I_n}-\mtx{\Vhat \Vhat^\top}).
    \end{align*}
    Compared to
    \begin{align*}
        \mtx{P}_\l^{-1} = \svhat \Vhat \mtx{\Lhat}^{-1} \Vhat^\top + (\mtx{I_n}-\mtx{\Vhat \Vhat^\top}),
    \end{align*}
    the only difference is the orthogonal factor $\mtx{\widehat{U}_M^\top \Vhat_M}$.

    Defining
    \begin{align*}
        \mtx{W}=  \Vhat (\mtx{\widehat{U}_M^\top \Vhat_M}) \Vhat^\top + (\mtx{I_n}-\mtx{\Vhat \Vhat^\top}),
    \end{align*}
    we obtain $\mtx{\widetilde{P}}^{-1} = \mtx{P}_\l^{-1}\mtx{W}$, and $\mtx{W}$ is orthogonal.
    
\end{proof}
The lemma shows that the two preconditioners differ only by a right orthogonal transformation. The following corollary formalizes the consequences.
\begin{corollary}[Convergence Equivalence]
    Let LSQR be applied to the systems preconditioned by $\mtx{P}_\l$ and $\mtx{\widetilde{P}}_\l$, producing iterates $\vct{y}_k = \vct{P}_\l \vct{x}_k$ and $\vct{\widetilde{y}}_k = \vct{\widetilde{P}}\vct{x}_k$. Then, at the $k$-th iterate,
    \begin{align*}
        \|\mtx{\widetilde{U}^{\top}}(\mtx{A}\mtx{\tilde{P}}_\l^{-1}\vct{\tilde{y}}_k-\vct{b})\|_2  = \|\mtx{\widetilde{U}^{\top}}(\mtx{A}\mtx{P}_\l^{-1} \vct{y}_k - \vct{b})\|_2,
    \end{align*}
    where $\mtx{\widetilde{U}}$ denotes the left singular vector matrix of $\mtx{AP}_\l^{-1}$.
\end{corollary}
\begin{proof}
    Let $\mtx{AP}_\l^{-1} = \mtx{\widetilde{U}\widetilde{\Sigma}\widetilde{V}^\top}$ be an SVD. Then, $\mtx{A\widetilde{P}}^{-1} = \mtx{\widetilde{U}\widetilde{\Sigma}(\mtx{W}^\top\widetilde{V})^\top}$ is also an SVD. Since $\vct{\widetilde{y}}_k = \mtx{W^\top}\vct{y}_k$, the coordinates in the right singular basis are unchanged, and the polynomial residual expression is identical:
    \begin{align*}
        \|\mtx{\widetilde{U}^{\top}}(\mtx{A}\mtx{\tilde{P}}_\l^{-1}\vct{\tilde{y}}_k-\vct{b})\|_2 &= \min_{p_k} \|\mtx{A}\mtx{\tilde{P}}_\l^{-1}(\mtx{W^\top} \mtx{\widetilde{V}})p_k(\mtx{\widetilde{\Sigma}}^2)(\mtx{W^\top} \mtx{\widetilde{V}})^\top\vct{\widetilde{y}_\ast}
        \|_2 \\
        &= \min_{p_k} \|\mtx{A}\mtx{P}_\l^{-1}\mtx{W}\mtx{W^\top} \mtx{\widetilde{V}}p_k(\mtx{\widetilde{\Sigma}}^2) \mtx{\widetilde{V}}^\top\mtx{W}\mtx{W^\top}\vct{y_\ast}\|_2 \\
        &= \min_{p_k} \|\mtx{A}\mtx{P}_\l^{-1} \mtx{\widetilde{V}}p_k(\mtx{\widetilde{\Sigma}}^2)\mtx{\widetilde{V}}^\top\vct{y_\ast}\|_2 = \|\mtx{\widetilde{U}^{\top}}(\mtx{A}\mtx{P}_\l^{-1} \vct{y}_k - \vct{b})\|_2.
    \end{align*}
\end{proof}
Hence, the SVD-free variant solves the same spectrally flattened problem, differing only by a rotation within the preconditioned subspace.

However, the equivalence between the SVD-free and SVD-based preconditioned problem is restricted to the unregularized case $\mu=0$. In the regularized setting, efficiently evaluating the inverse of regularized spectrum without the SVD is nontrivial. To work around this, we instead form a CUR approximation of the regularized matrix $\mtx{A}_\mu$ and construct the SVD-free preconditioner based on this approximation. This modification breaks the exact equivalence between the SVD-based and SVD-free preconditioned problems. Nevertheless, empirical results indicate that this difference has no noticeable impact in practice (see \cref{sec:numstud_svdfree}).

\subsection{Adaptivity criteria}
To realize the rank and schedule adaptivity described in \cref{sec:adaptivity}, we introduce three criteria: (i) CUR approximation tolerance controlling rank growth, (ii) a re-preconditioning condition controlling when the preconditioner is updated, and (iii) a dynamic stopping rule determining the duration of each LSQR phase. These criteria ensure that the computational effort is aligned with the spectral structure and the requested accuracy of the problem.

\subsubsection{Rank adaptivity: CUR tolerance}

Rank growth is governed by monitoring the spectral error of the CUR approximation. Since LSQR convergence is controlled by the largest unresolved singular component (\cref{sec:section2}), the spectral norm is the appropriate metric in the preconditioning context.

For computational efficiency, we estimate the spectral norm of the sketched rank-$\l$ residual $\mtx{E}^{\mathrm{row}}$ using the randomized bound in~\cite[Lemma 4.1]{halko2011finding}:
\begin{align}\label{eq:2normestimate}
    \|\mtx{E}^{\mathrm{row}}\|_2 \leq 10 \sqrt{2/\pi} \max_{i = 1, \ldots, r} \|\mtx{E}^{\mathrm{row}}\vct{w}_i\|_2 =: \rho,
\end{align}
where $\vct{w}_i$ are i.i.d. standard Gaussian vectors. This bound holds with probability at least $1-10^{-r}$; in our implementation we take $r=10$. We avoid power-iteration-based estimators~\cite{liberty2007randomized}, as they tend to underestimate the spectral norm and may therefore lead to rank underestimation. At each iteration, we monitor the error estimate $\rho$ and increase the rank until the \emph{termination condition} $$\rho \leq \varepsilon_{\mathrm{cur}}$$ is satisfied for the prescribed error tolerance $\varepsilon_{\mathrm{cur}}$. In practice, $\varepsilon_\mathrm{cur}$ should be chosen relative to the regularization level $\mu$; a robust choice is $\varepsilon_{\mathrm{cur}}\in[30\mu,100\mu]$, which avoids truncating significant spectral components while preventing unnecessary rank growth.


\subsubsection{Schedule adaptivity: Re-preconditioning and stopping criteria}\label{sec:adaptivitycriteria}
While the spectral tolerance determines \emph{what} to precondition (rank adaptivity), schedule adaptivity determines \emph{when} to update the preconditioner and how long each LSQR phase should continue. Two criteria govern this process.

\paragraph{Re-preconditioning criterion}
Let $d_{\mathrm{cur}}:= \rho - \varepsilon_{\mathrm{cur}}$ denote the slackness of the CUR approximation relative to desired accuracy at the previous preconditioning step. We trigger a new preconditioning phase only when the improvement in the CUR approximation is sufficiently significant, according to $$d_{\mathrm{cur}} / (\rho - \varepsilon_{\mathrm{cur}}) \geq \nu_{\mathrm{\mathrm{prec}}},$$where $\nu_{\mathrm{\mathrm{prec}}}$ is the re-preconditioning tolerance.
This condition prevents unnecessary reconstruction of the preconditioner when additional rank yields only marginal benefit. 
In practice, the method is relatively insensitive to this parameter, and a moderate value such as $\nu_{\mathrm{prec}} \approx 10$ provides a good balance between the cost of updating the preconditioner and the expected gain in spectral flattening.

\paragraph{Dynamic stopping criteria}
Within each preconditioned LSQR phase, it is important to fully exploit the current preconditioner before increasing the target rank. Increasing the rank generally improves conditioning by flattening additional singular values, but re-preconditioning too early wastes computational effort. To balance this trade-off, we monitor the convergence behavior of LSQR and adaptively decide when to terminate the current phase.

Let $\bar{\phi}_j$ denote the estimated residual at the $j$th LSQR iteration, computed as an intermediate scalar arising from the Golub–Kahan bidiagonalization process~\cite[Section~5]{paige1982lsqr}, with $\bar{\phi}_{0}$ denoting the initial residual. Define\begin{align*}
    \mathrm{cvgrate}_j = \log(\bar{\phi}_{j-1}/\bar{\phi}_{j}) \qquad \text{and} \qquad
    \mathrm{cvgdiff}_j = \bar{\phi}_{j-1} - \bar{\phi}_{j}.
\end{align*} We terminate the current LSQR phase when 
\begin{align*}
    \frac{\mathrm{cvgrate}_1}{\mathrm{cvgrate}_j} > \nu_{\mathrm{\mathrm{lsqr}}}\qquad \text{or} \qquad
    \mathrm{cvgdiff}_j < \svhat_\l,
\end{align*}where $\nu_{\mathrm{\mathrm{lsqr}}}$ is a prescribed tolerance and $\svhat_\l$ is the smallest singular value of the current CUR approximation. Empirically, values in the range $\nu_{\mathrm{lsqr}}\in[100,200]$ provide a robust choice: smaller values may trigger premature termination, while larger values can lead to inefficient LSQR iterations with slow progress.

These criteria are designed to detect a slowdown in the LSQR convergence. When the rate deteriorates sufficiently relative to the initial rate, or when the residual decrease falls below the smallest resolved singular value scale $\svhat_\l$, this indicates that the benefit of the current preconditioner has likely been exhausted and that further rank expansion is needed.

The dynamic LSQR termination rule is applied alongside the standard LSQR stopping criteria, such as monitoring the gradient of residual and the projected residual~\cite{epperly2024fastforwardstablerandomized,chang2009stopping}, and is deactivated in the final LSQR phase once the CUR approximation has reached its prescribed tolerance.

Together, the approximation error tolerance, re-preconditioning condition, and dynamic LSQR stopping rules provide adaptive control of our algorithm. The rank of the preconditioner grows only when the CUR approximation improves meaningfully, and LSQR phases are neither prematurely terminated  nor unnecessarily prolonged. This ensures that computational effort scales naturally with the spectral properties of the problem and the requested accuracy.

\subsection{Full \mainalg{} procedure}
We now summarize the complete adaptive solver in \cref{alg:aplls}. The method integrates rank-adaptive CUR construction with adaptively scheduled preconditioned-LSQR phases for solving the regularized least-squares problem.

The algorithm consists of three main components executed within an outer rank-expansion loop. Firstly, the CUR factors are incrementally augmented 
and the spectral residual estimate $\rho$ is updated (Part 1). Secondly, the necessary factorizations of $\mtx{R}$ are maintained to update the preconditioner efficiently (Part 2). Thirdly, whenever the re-preconditioning criterion is satisfied, a new preconditioner is formed and LSQR is invoked with a warm 
start and dynamic stopping control (Part 3).

The process terminates once the CUR approximation meets the prescribed spectral tolerance $\varepsilon_{\mathrm{cur}}$. In this way, rank growth, preconditioner updates, and Krylov iterations are coordinated adaptively.

\begin{algorithm}
    \caption{\textsl{\mainalg{}: detailed implementation}}
    \label{alg:aplls}
    \begin{algorithmic}[1]
        \Require $\mtx{A}\in\F^{m\times n}$; $\vct{b}\in\F^{m}$; regularization $\mu>0$; block size $b$; CUR tolerance $\varepsilon_{\mathrm{cur}}$; LSQR tolerance $\varepsilon_{\mathrm{lsqr}}$; re-preconditioning tolerance $\nu_{\mathrm{prec}}>1$; LSQR dynamic stopping tolerance $\nu_{\mathrm{lsqr}}>1$.
        \Statex
        \State \textbf{Initialise:}
        \State $\mtx{A}_\mu = [\mtx{A}; \mu \mtx{I}_n]$, $\vct{b}_{\mathrm{aug}} = [\vct{b};\vct{0}_{n \times 1}]$
        \State $\mtx{S}\gets\texttt{sparsesign}(\lceil1.1 b\rceil, m, 8)$ \Comment{$b\times m$ random embedding}
        \State $\mtx{E}^{\mathrm{row}}\gets\mtx{SA}$ \Comment{compute sketch once and reuse}
        \State $\vct{I}=\vct{J}=\emptyset,\quad \mtx{C}=\mtx{U}=\mtx{R}=\emptyset,\quad$ $\mtx{Q_R}=\mtx{T_R}=\emptyset$
        \State $d_{\mathrm{cur}}=\texttt{Inf}$,\quad $\vct{x}_0=\vct{0}$
        \Statex
        \Repeat \Comment{rank increases by $b$ each iteration}
            \State \textbf{Part 1: Update CUR}
            \State $[\vct{I}, \vct{J}] \gets \texttt{augmentCUR}(\mtx{A}, \mtx{C}, \mtx{U}, \mtx{R}, \mtx{E}^{\mathrm{row}}, \vct{I}, \vct{J})$ \Comment{use \cref{alg:cur_augment}}
            \State $\mtx{C}\gets\mtx{A}(:,\vct{J}),\quad$$\mtx{U}\gets\mtx{A}(\vct{I},\vct{J})^\dagger,\quad$$\mtx{R}\gets\mtx{A}(\vct{I},:)$
            \Comment{use \texttt{qr} or \texttt{lu} for $\mtx{U}$}
            \State $\mtx{E}^{\mathrm{row}}\gets\mtx{SA}-\mtx{S}\,\mtx{C}\,\mtx{U}\,\mtx{R}$ \Comment{use $\mtx{SC}$ from $\mtx{SA}$}
            \State $\rho\gets \text{estimate of }\|\mtx{E}^{\mathrm{row}}\|_2$ \Comment{estimated by \cref{eq:2normestimate}}
            \Statex
            \State \textbf{Part 2: Update Factorizations}
            \State $\vct{I_+} \gets \vct{I}(\texttt{end}-b +1, \texttt{end})$ \Comment{updated rows}
        \State $[\mtx{Q_R},\mtx{T_R}]\gets\texttt{augmentQR}(\mtx{Q_R},\mtx{T_R},\mtx{A}(\vct{I_+},:)^\top)$ \Comment{use \cref{alg:augmentqr}}
            \Statex
            \State \textbf{Part 3: Update Preconditioner and run LSQR}
            \If{$d_{\mathrm{cur}}/(\rho-\varepsilon_{\mathrm{cur}}) \ge \nu_{\mathrm{prec}}$ \textbf{or} $\rho \le \varepsilon_{\mathrm{cur}}$} \Comment{re-preconditioning}
                \State $\l \gets \texttt{length}(\mtx{I})$
                \State $\mtx{T_C}\gets\texttt{chol}(\mtx{C}^\top\mtx{C})$ \Comment{use sketched Cholesky QR}
                \State $\mtx{P}_{\l,\mu}\gets$ \textbf{Part 2} of \texttt{CURPreconditioner} in \cref{alg:preconditioner}
                \State $\nu_{\mathrm{lsqr}}\gets \nu_{\mathrm{lsqr}} + \texttt{Inf}\cdot\mtx{1}_{\rho \le \varepsilon_{\mathrm{cur}}}$ \Comment{dynamic stopping option}
                \State $\vct{x}_0\gets\texttt{\texttt{lsqr}}(\mtx{A}_\mu,\vct{b}_{\mathrm{aug}},\mtx{P}_{\l,\mu},\vct{x}_0, \varepsilon_{\mathrm{lsqr}},\nu_{\mathrm{lsqr}})$ \Comment{LSQR with warm start $\vct{x}_0$}
                \State $d_{\mathrm{cur}}\gets \rho-\varepsilon_{\mathrm{cur}}$ \Comment{error–tolerance gap}
            \EndIf
        \Until{$\rho \le \varepsilon_{\mathrm{cur}}$}
        
        \State \Return $\vct{x}_0$
    \end{algorithmic}
\end{algorithm}

\subsection{Computational complexity}

The computational cost of the proposed algorithm is dominated by two components: (i) factorizations required to construct and update the low-rank preconditioner, and (ii) matrix–matrix multiplications involving the full dimensions $m$ and $n$.

Let $\l_{\textrm{final}}$ be the final target rank of our algorithm. Then, the construction of a rank-$\l_{\textrm{final}}$ CUR approximation costs $\mathcal{O}(mn+(m+n)\l_{\textrm{final}}^2)$, while requiring only $\mathcal{O}(\l_{\textrm{final}}(m+n))$ storage. Also, at each precondition-solve phase, given CUR factors with target rank $\l$, constructing a rank-$\l$ preconditioner costs $\mathcal{O}((m+n+\l)\l^2)$, while storage and application require only $\mathcal{O}(n\l)$. The main bottleneck is the $\mathcal{O}(\l^3)$ SVD required in each preconditioner update, which we eliminate using the SVD-free variant (\cref{sec:svdfreepreconditioner}), reducing the construction cost at the expense of a modest $\mathcal{O}(\l^2)$ increase in application cost.

Additional factorization costs (QR updates, CUR core matrix inversion, and LUPP) are mitigated through incremental updates and the use of efficient dense factorizations~\cite{balabanov2022randomized,fukaya2014,stewart1998}. Matrix–matrix multiplications, while nominally costly, are implemented using BLAS-3 kernels and further reduced in the presence of sparsity. The sketching cost is controlled via sparse embeddings.

A detailed breakdown of all computational components and their associated costs is provided in \cref{app:complexity}.

\section{Experimental Framework}\label{sec:expframework}

\subsection{Roadmap of the numerical experiments}
The remainder of the paper is organized as follows. \cref{sec:expframework} describes the experimental framework, including the test problems, implementation details, and evaluation metrics used throughout the numerical studies. \cref{sec:numericalstudies} investigates the proposed CUR-based preconditioning framework in detail, focusing on the roles of rank adaptivity, schedule adaptivity, and the SVD-free variant, and on how these design choices affect convergence and setup cost. Finally, \cref{sec:benchmark} benchmarks \mainalg{} against representative randomized preconditioned solvers on a range of large-scale regularized least-squares problems, highlighting its robustness, scalability, and performance in challenging sparse and ill-conditioned regimes.

\subsection{Implementation details and environment}
The numerical experiments were conducted on regularized linear least-squares problems of the form \eqref{eq:regularizedLLS} using synthetic test sets generated to address different spectral and structural properties.

All algorithms are implemented in \texttt{MATLAB R2025b} on Mac Mini (Apple M4, 10-core CPU, 16GB RAM) and used MATLAB's built-in multithreaded BLAS/LAPACK. Wall-clock time is measured using \texttt{tic/toc}. Each experiment was repeated over $5$ independent trials with different random seeds; variations across trials are negligible and omitted for clarity.  All code and scripts to reproduce figures will be made available at \jh{github link}.

\subsection{Test sets}

\subsubsection{Label vector}
We consider three problem settings that differ in their consistency properties, determined by the construction of the right-hand side $\vct{b}$. In the \textit{consistent-x} setting, we first draw a random Gaussian vector $\vct{x}_\ast$ and define the right-hand side as $\vct{b} = \mtx{A}\vct{x}_\ast+\vct{e}$, where $\vct{e}$ is a noise vector chosen to be orthogonal to the column space of $\mtx{A}$ (enforced by twice orthogonalization). This construction ensures that $\vct{x}_\ast$ is the exact minimizer of $\min_{\vct{x}}\|\mtx{A}\vct{x}-\vct{b}\|_2$ with a modest-norm solution, $\|\vct{x}_\ast\|_2\approx \sqrt{n}$. Consequently, the computed residual is bounded below by $\|\vct{e}\|_2 = \|\mtx{A}\vct{x}_\ast-\vct{b}\|_2$. In the \textit{consistent-b} setting, the right-hand side is sampled directly from the column space of $\mtx{A}$ by generating $\vct{b} = \texttt{orth(A) * randn(n,1)} $. Thus, $\vct{b}$ lies exactly in $\mathrm{range}(\mtx{A})$. In this case the least-squares problem is consistent, although the corresponding solution will have large norm when $\mtx{A}$ is highly ill-conditioned and the regularization parameter $\mu$ is small.

\subsubsection{Test matrix}

For the test matrix $\mtx{A}\in \F^{m\times n}$, we consider a range of structural and spectral properties known to influence the performance of sketching-based preconditioners and iterative solvers. In particular, we vary: (i) the sparsity, (ii) the spectral decay of $\mtx{A}$ (sharp versus smooth), (iii) the dimension ratio $m/n$, and (iv) the coherence of the matrix.

The coherence of a 
matrix $\mtx{A}$ quantifies how uniformly its information is distributed across its rows, and is known to affect the effectiveness of sampling-based algorithms. In particular, coherence has influence on the performance of CUR approximations, which explicitly sample rows and columns of the matrix, and is therefore especially relevant when comparing against methods such as Blendenpik. 
Following standard definitions \cite{mahoney2011randomized}, the coherence of a matrix $\mtx{A}$ is defined in terms of the row norms of its left singular vectors. Let $\mtx{A}=\mtx{U}\mtx{\Sigma}\mtx{V}^\top$ be the singular value decomposition of $\mtx{A}$. The coherence $\nu(\mtx{A})$ is given by 
\begin{align*}
    \nu(\mtx{A}) = \max_{i \in [m]} \|\mtx{U}_i\|_2,
\end{align*}
where $\mtx{U}_i$ denotes the $i$-th row of $\mtx{U}$. High coherence is well known to degrade the performance of uniform sampling. However, the data-aware sampling strategies, such as the LUPP-based selection employed in our CUR approximation, can instead exploit coherence to identify informative rows and columns more effectively, resulting in a higher accuracy CUR approximation close to truncated SVD.

The synthetic matrices are constructed to provide precise control over spectral decay and coherences enabling systematic tests of the algorithm.

\begin{enumerate}
    \item \textbf{Dense synthetic dataset.} The matrix $\mtx{A} = \mtx{U\Sigma V}^\top$ is explicitly constructed via singular value decomposition, where the singular values are prescribed by $\text{diag}(\mtx{\Sigma})$. The left and right singular vectors, $\mtx{U}$ and $\mtx{V}$, are generated to reflect different levels of coherences:
    \begin{enumerate}
        \item \textit{Incoherent}: $\mtx{U}\in\F^{m \times n}$ and $\mtx{V}\in\F^{m \times n}$ are random orthogonal matrices from the Haar distribution.
        \item \textit{Coherent}: With a matrix of all ones, $\mtx{J}_{m,n} \in \F^{m \times n}$, define
        $$
            \mtx{U} = \mathrm{orth}\left(\begin{bmatrix} \mtx{I}_n \\ \mtx{0} \end{bmatrix} + 10^{-8} \cdot \mtx{J}_{m,n}\right), \quad
            \mtx{V} = \mathrm{orth}\left(\mtx{I}_n + 10^{-8} \cdot \mtx{J}_{m,n}\right).
        $$
    \end{enumerate}

    \item \textbf{Sparse synthetic dataset.} Let $\mtx{B} = \texttt{sprandn}(m,n,0.01)$, where \texttt{sprandn()} is a {MATLAB} function that generates an $m \times n$ sparse matrix with approximately $0.01mn$ non-zero entries drawn independently from the standard normal distribution. Let $\mtx{D}$ be a diagonal matrix whose diagonal entries are drawn independently from $\mathcal{N}(0,1)$. We then construct sparse synthetic datasets with varying levels of coherence using the coherency factor $f$ as follows:
    \begin{align*}
        \mtx{\widetilde C} &= \mtx{D}^f \mtx{B},\\
        \mtx{C_{:,j}} &= \frac{\mtx{\widetilde C_{:,j}}}{\|\mtx{\widetilde C_{:,j}}\|_2}, \quad j = 1, \ldots,m, \\
        \mtx{A} &= \mtx{C}\mtx{\Sigma}
    \end{align*}
    where $f=0$ for incoherent matrices and $f=20$ for coherent matrices, respectively. 
    
\end{enumerate}

For spectral decays, we deal with two different types of decay: sharp decay and smooth decay. In the sharp decay case, the singular values of the matrices with condition number $10^7$ and $10^{15}$ are prescribed as
\begin{align*}
    \text{diag}(\mtx{\Sigma}) = \texttt{[logspace(2,-2,0.2*n),logspace(-4.8,-5,0.8*n)]},\\
    \text{diag}(\mtx{\Sigma}) = \texttt{[logspace(2,-2,0.2*n),logspace(-12,-13,0.8*n)]},
\end{align*}
respectively. Here, $\texttt{logspace}(a, b, m)$ denotes $m$ numbers spaces logarithmically between $10^a$ and $10^b$.
In the smooth decay case, the matrix is highly ill-conditioned, with condition number $10^{15}$, and exhibits a root-exponential spectral decay.

\subsection{Evaluation metrics}
In the following sections (\cref{sec:numericalstudies} and \cref{sec:benchmark}), we assess the performance of the algorithm using various evaluation metrics. Specifically, we examine the quality of the preconditioned singular values (leading and trailing singular values) to evaluate spectral improvement and the relative residual norm versus wall-clock time to assess overall computational efficiency.

For regularized problems, we compare the relative residual $\|\mathbf{A}\mathbf{x}_k - \mathbf{b}\|_2 / \|\mathbf{b}\|_2$ against the optimal relative residual $\|\mathbf{A}\mathbf{x}_\mu - \mathbf{b}\|_2 / \|\mathbf{b}\|_2$. For unregularized problems, we similarly compare $\|\mathbf{A}\mathbf{x}_k - \mathbf{b}\|_2 / \|\mathbf{b}\|_2$ with the optimal relative residual $\|\mathbf{A}\mathbf{x}_* - \mathbf{b}\|_2 / \|\mathbf{b}\|_2$. Here, $\mathbf{x}_\mu$ and $\mathbf{x}_*$ denote the exact minimizers of$$\min_{\mathbf{x}} \|\mathbf{A}_\mu \mathbf{x} - \mathbf{b}_{\text{aug}}\|_2 \quad \text{and} \quad \min_{\mathbf{x}} \|\mathbf{A}\mathbf{x} - \mathbf{b}\|_2,$$respectively, with solution norms satisfying $\|\mathbf{x}_\mu\|_2 \approx \sqrt{n}$ and $\|\mathbf{x}_*\|_2 \approx \sqrt{n}$.

\section{Numerical Studies}\label{sec:numericalstudies}

This section investigates the behavior of the proposed CUR-based preconditioning framework through a series of targeted numerical experiments. Our aim is not only to evaluate performance, but also to better understand how the key design components---namely rank adaptivity, schedule adaptivity, and the choice of implementation---affect convergence and computational cost.

We begin by isolating the roles of the two forms of adaptivity and examining their impact on convergence across different spectral regimes. We study the sensitivity of the method to target rank selection, highlighting how the adaptive framework mitigates performance degradation arising from under- or over-preconditioning. Then, we compare different implementations of the preconditioner, including the SVD-based and SVD-free variants, and assess their impact on computational cost and convergence behaviur. Detailed implementation aspects are described in \cref{app:numericalstudies_detail}.

Unless stated otherwise, all experiments use the following baseline configuration. We consider a \emph{consistent-x} setting with $\mtx{A}$ chosen as a dense incoherent matrix. The CUR augmentation block size is $b = n/50$, and the CUR tolerance is set to $\varepsilon_{\mathrm{cur}} = 30\mu$, with regularization parameter $\mu = 10^{-4}$. Preconditioner updates are triggered using $\nu_{\mathrm{prec}} = 10$.
For LSQR, we use a tolerance of $\varepsilon_{\mathrm{lsqr}} = 10^{-10}$ together with a dynamic stopping parameter $\nu_{\mathrm{lsqr}} = 100$. Sketching uses sparse sign embeddings with $\xi=8$ nonzeros per column; alternative sketching schemes are evaluated in \cref{sec:sketch}.

\subsection{Levels of adaptivity}
The proposed framework incorporates two distinct forms of adaptivity: (1) \emph{rank adaptivity}, achieved via iterative CUR approximation, and (2) \emph{schedule adaptivity}, achieved by alternating between preconditioning and LSQR phases.
\yn{YN; repetitive. Shrink for paper}
To isolate the role of each adaptivities, we consider three different algorithmic variants:
\begin{itemize}
    \item \textbf{Fixed-rank single-shot preconditioning}, where a prescribed target rank is used, the preconditioner is constructed once, and only a single solving phase is performed.
    \item \textbf{Rank-adaptive single-shot preconditioning}, where the target rank is determined adaptively, but the preconditioner is still constructed once and applied throughout a single solving phase.
    \item \textbf{Fully-adaptive preconditioning}, where both the target rank and the preconditioning schedule are adaptive, and preconditioning is interleaved with multiple LSQR solving phases.
\end{itemize}

Provided that an appropriate rank is given, all three variants are robust in the sense that they reliably reach the optimal residual much faster than unpreconditioned LSQR. Their differences lie instead in efficiency, sensitivity to rank misestimation and problem settings, and how computational effort is distributed over the course of the solve.

\subsection{Sensitivity to target rank selection}\label{sec:numstud_rankadaptive}

We first examine the effect of rank selection on solver performance, focusing on the impact of under- and overestimation of the numerical rank. Rank misestimation directly affects the quality of the CUR preconditioner and, in turn, the convergence behavior of LSQR.

Assuming a numerical rank of $0.2n$, we vary the target rank $\l=0.1n, 0.2n, 0.3n$ and run the fixed-rank single-shot variant of the algorithm, in which the CUR approximation and preconditioner are constructed once and used for a single LSQR phase. \cref{fig:restimes_rankest} shows the resulting   convergence versus wall-clock time, together with the corresponding preconditioned spectra.

\input{plot_tex/num_stud_rankest2.tex}

Preconditioning with underestimated target rank yields only transient improvement in convergence. While a small rank yields a short setup time and early progress, insufficient spectral flattening causes convergence to deteriorate rapidly. This effect is especially pronounced in the consistent-x case, where convergence is strongly influenced by the relative conditioning of the dominant singular components.

When the target rank is well matched to the numerical rank, the dominant singular values are effectively flattened and the remaining spectrum are well conditioned by regularization, resulting in rapid and sustained convergence. As shown in \cref{fig:restimes_rankest}, even when $\|\mathbf{e}\|_2 = 0$, the residual does not converge to zero. This is due to the application of regularization, which renders the problem effectively inconsistent, and the inherent ill-conditioning of $\mathbf{A}$. Specifically, in the consistent-b setting, $\vct{b}$ has significant components aligned with singular vectors corresponding to very small singular values, which limits the attainable residual norm due to finite-precision effects.

Overestimation has more problem-dependent effects. In consistent-b problems and in cases with smooth spectral decay, convergence remains robust and may even from the overestimation, reflecting the reduced condition number after preconditioning. In contrast, in th consistent-x case with sharp spectral decay, overestimation can significantly degrade convergence despite favorable conditioning. This behavior is examined in more detail below.
In addition, \cref{fig:restimes_rankest} shows that setup cost increases rapidly with target rank. Overestimation therefore often introduces substantial computational overhead without proportional gains in convergence. These results on misestimation of target ranks motivate the need for adaptive rank selection.




\subsubsection{Over-preconditioning and problem dependence}\label{sec:overpreconditioning}

The degradation observed under over-preconditioning in the consistent-x case can be attributed to the destruction of the spectral gap between dominant and trailing singular values. When the target rank exceeds the numerical rank in matrices with sharp spectral decay, the preconditioned spectrum becomes tightly clustered, impairing LSQR's ability to resolve dominant components efficiently.

In consistent-x problems, where the transformed solution $\vct{y}_\ast = \mtx{P}_{\l,\mu}\vct{x}_\ast$ is strongly aligned with the dominant singular subspace, this loss of spectral separation leads to poor Krylov subspace alignment and slow convergence. In contrast, in matrices with no sharp drop in spectral decay, or in consistent-b setting, the transformed solutions are approximately uniformly distributed across singular directions, and the convergence is driven primarily by reduced conditioned number rather than spectral gaps. Additional details are provided in \cref{sec:overpreconditioning_detail}.

This motivates the adaptive preconditioning and solving, where early LSQR phases operate on a preconditioned system that preserves sufficient spectral separation, while later phases refine the solution using warm starts. 

This analysis also explains why under-preconditioning is particularly detrimental in the consistent-x setting, where the convergence depends critically on relative conditioning among the dominant singular values, which remains poor when the rank is underestimated.

\jh{At the same time, the adaptive CUR approximation naturally limits unncessary rank growth, introducing additional preconditioning only when it is likely to be beneficial.}

\subsection{Fully adaptive preconditioning: Robustness and low setup cost}\label{sec:adaptiveprec_sensitivity}

The analysis above highlights the sensitivity of single-shot strategies to rank selection. Although rank-adaptive single-shot preconditioning can reliably construct an effective preconditioner, it may incur a large upfront cost when the chosen rank exceeds what is required for the desired accuracy. In practice, rank estimation is affected by hyperparameter choices, spectral perturbations arising from the CUR approximation, and inaccuracies in sketched error estimates, motivating the need for a more robust strategy that is less sensitive to target-rank misestimation.

The fully adaptive algorithm addresses this challenge by alternating between preconditioning and warm-started LSQR phases. This allows computational effort to be distributed over the solve, enabling LSQR to start immediately and achieve early reduction of the residual, while progressively refining the preconditioner only when necessary.

To isolate the effect of adaptive preconditioning, we compare the rank-adaptive single-shot preconditioning variant with fully-adaptive method. \cref{fig:acctimes8_adaptiveness2} shows that while the rank-adaptive single-shot variant guarantees robustness, it may incur a large setup cost before any convergence is observed. In contrast, the fully adaptive method amortises this cost by interleaving LSQR iterations with incremental preconditioner updates. This allows LSQR to start immediately and reach moderate accuracy early. These results indicate that adaptivity is not strictly necessary for ultimate convergence speed, but is essential for controlling setup cost and enabling rapid early progress when the appropriate preconditioning strength is not known in advance or when the requested accuracy is moderate.

\input{plot_tex/num_stud_singleshot}

\cref{fig:restimes_rankestfullyadaptive} further demonstrates that fully adaptive preconditioning significantly reduces sensitivity to rank misestimation. Underestimation still leads to slower convergence, but overestimation no longer causes severe degradation, even in the consistent-$\vct{x}$ case with sharp spectral decay. This robustness arises from warm-starting: later LSQR phases begin from iterates already well aligned with the solution subspace, limiting the adverse impact of over-preconditioning.

Overall, these results highlight the benefit of combining rank adaptivity with adaptive preconditioning schedules. The fully adaptive variant provides a robust default that achieves fast early progress, avoids excessive setup time, and remains reliable across problem types and spectral regimes. In practice, we therefore recommend using small block sizes and allowing the target rank to grow adaptively, which limits sketching cost and prevents premature over-preconditioning.

Nevertheless, even though overestimation of the target rank no longer leads to severe degradation in convergence, it remains undesirable from a computational standpoint. Excessive rank growth introduces unnecessary re-preconditioning steps, which delay the final LSQR phase and increase the overall runtime. Since the cost of preconditioning grows rapidly with the target rank, primarily due to dense SVD computations, avoiding excessive rank is important for maintaining efficiency, even when convergence robustness is preserved.

\subsection{SVD-free preconditioner}\label{sec:numstud_svdfree}

We compare the SVD-based and SVD-free constructions introduced in \cref{sec:svdfreepreconditioner}. The SVD-free variant replaces the spectral inverse obtained via SVD with a two–triangular-solve formulation, removing the cubic cost associated with repeated SVD updates.

As shown in \cref{fig:restimes_svdfree}, both variants provide comparable preconditioning quality across regularization levels and spectral regimes. However, the SVD-based update becomes increasingly expensive as the target rank grows, whereas the SVD-free construction maintains a nearly stable update cost throughout the adaptive process. Thus, the SVD-free approach eliminates the SVD bottleneck without sacrificing convergence.

\input{plot_tex/num_stud_svdfree}

While the SVD-based and SVD-free preconditioners are mathematically equivalent in the unregularized setting, their constructions differ slightly for regularized problems. The SVD-based approach applies iterative CUR to $\mtx{A}$ and incorporates $\mu$ explicitly into the preconditioner, whereas the SVD-free variant applies CUR directly to $\mtx{A}_\mu$, since there is no straightforward way to form a regularized SVD-free preconditioner explicitly. Consequently, in the regularized experiments of \cref{fig:restimes_svdfree}, the selected target ranks may differ between the two constructions.

\section{Numerical Experiments}\label{sec:benchmark}
This section evaluates the performance of \mainalg{} in comparison with existing randomized preconditioned solvers for large-scale regularized linear least-squares problems. The goal of the experiments is not only to compare time-to-residual, but also to assess robustness, scalability, and overall solver behavior across challenging problem regimes. The experiments are designed to assess how effectively each method balances preconditioner construction cost against solver convergence, particularly when the numerical rank and spectral structure of the problem are not known a priori.

The experiments are structured as progressive stress tests. We first examine general convergence behavior, then extreme ill-conditioning and vanishing regularization, and finally large-scale sparse problems. The experiments are done on broad test sets covering variations in problem consistency, spectral decay, matrix coherence and dimension. 

\subsection{Methods compared}
We compare two variants of our algorithm---\mainalg{} (SVD-based) and \mainalg{}-sf (SVD-free)---against a set of representative baseline methods:
\begin{itemize}
    \item LSQR: Unpreconditioned LSQR, serving as a baseline iterative solver.
    \item Nystr\"om PCG (NPCG): Adaptive low-rank Nystr\"om preconditioning applied to CG on normal equations.
    \item Blendenpik: Randomized QR-based preconditioning for LSQR.
\end{itemize}
Unless stated otherwise, all methods are tuned using recommended or commonly used parameter settings, and are run until either the target residual is reached or the residual convergence is stationary. Both \mainalg{} and Nystr\"om PCG use a block size (i.e., initial sketch dimension) of $n/50$. Nystr\"om PCG progressively doubles the sketch dimension until the approximation meets the prescribed tolerance. For fair comparison, we set the sketch dimension in Blendenpik to $\gamma n$ with $\gamma = 1.2$, rather than the commonly recommended range $\gamma \in [3, 5]$.
\subsection{Fast initial convergence via adaptive preconditioning}
We begin by examining relative residual versus wall-clock time across representative problem settings.

\input{plot_tex/compare_restime}

\cref{fig:restimes_varioussettings} shows that \mainalg{} consistently achieves rapid initial residual reduction while maintaining competitive long-term convergence. Unpreconditioned LSQR often stagnates and fails to attain higher accuracy, and static preconditioning approaches—such as Nystr\"om PCG and Blendenpik—incur
substantial upfront cost for preconditioner construction before Krylov convergence begins. In contrast, \mainalg{} interleaves low-rank preconditioning updates with warm-started LSQR phases. This enables immediate progress and allows computational effort to scale naturally with the requested accuracy. The behaviour is consistent across coherent and incoherent matrices, consistent and inconsistent systems, varying aspect ratios, and both sharp and smooth spectral decay (see Appendix~\ref{app:benchmark_general}). 

Our algorithm remains competitive even in highly overdetermined regimes (\cref{fig:restimes_varioussettings}c), where Blendenpik is typically strong due to its sketch size scaling primarily with the column dimension. One might expect row- and column-sampling preconditioners to struggle on coherent matrices. However, (\cref{fig:restimes_varioussettings}d) indicates that \mainalg{} is largely insensitive to coherence. This robustness stems from its data-aware pivoted selection strategy: rather than sampling uniformly, it identifies informative rows and columns, thereby exploiting coherence instead of suffering from it. Performance remains similarly stable under different spectral decay patterns.

Additionally, we can observe that Blendenpik converges slowly even after full preconditioning. As discussed in \cref{sec:overpreconditioning}, this effect is likely linked to the loss of spectral separation induced by full-rank preconditioning. Nystr\"om PCG avoids this specific degradation by aggressively clustering the spectrum (often achieving $\kappa(A_\mu P_\mu^{-1}) \approx 1.1$ (see \cref{fig:restimes_varioussettings}a)), but at the cost of a substantial setup phase. Although \mainalg{} typically produces less extreme conditioning, its adaptive scheduling compensates in practice.

\subsection{Robustness under extreme ill-conditioning}

We next evaluate the performance of \mainalg{} under extreme ill-conditioning, where the condition number reaches $\kappa(\mtx{A}) \approx 10^{15}$, with varying levels of regularization.

Under moderate ill-conditioning ($\kappa=10^7$), all preconditioners perform well (\cref{fig:restimes_ill}a). However, in the extreme case (\cref{fig:restimes_ill}b), Nystr\"om PCG fails due to its reliance on the Gram matrix $A^\top A$, whose condition number satisfies $\kappa(A^\top A) = \kappa(A)^2$. This squaring effect makes the method highly sensitive to extreme ill-conditioning.

\input{plot_tex/compare_restime_ill}

In \cref{fig:restimes_ill}, the regularization parameter $\mu$ is chosen so that $\mathrm{rank}(\mtx{A},\mu) \approx 20n$. Since regularization itself alleviates ill-conditioning, we now examine how each method behaves when regularizatiion parameter is gradually reduced. \cref{fig:restimes_reg} reports experiments on $6000 \times 5000$ matrices with sharp smooth decay and condition number $\kappa = 10^{15}$, using regularization levels $\mu = 10^{-4}, 10^{-6}, 10^{-8}$, as well as the unregularised case. For the unregularized problem, we terminate CUR at tolerance $\varepsilon_{\mathrm{cur}} = 3 \cdot 10^{-7}$ to enforce a low-rank preconditioner; otherwise the algorithm will be ran until the full-rank preconditioner is constructed. In the regularized cases, we use $\varepsilon_{\mathrm{cur}} = 30\mu$, consistent with earlier experiments.

The results reveal a clear hierarchy of robustness. Nystr\"om PCG operates only when the regularization is sufficiently large; otherwise it fails to converge or encounters numerical breakdown. Blendenpik performs reliably down to $\mu = 10^{-8}$ but fails in the unregularized setting. In contrast, \mainalg{} remains robust even without regularization and consistently achieves rapid convergence to the optimal residual across all levels of regularization. Because it avoids forming normal equations and progressively modifies only dominant singular values, it is significantly less sensitive to extreme spectral conditioning.

\input{plot_tex/compare_restime_reg}

\subsection{Sparse large-scale performance}

A key practical advantage of the CUR-based construction is that the factors $\mtx{C}$ and $\mtx{R}$ inherit the sparsity pattern of the original matrix $\mtx{A}$, since they consist of selected columns and rows of $\mtx{A}$. As a result, both the construction and the application of the preconditioner can exploit sparsity directly. In particular, matrix–vector products involving the preconditioner scale with the number of nonzeros in the selected rows and columns, rather than with the ambient dimensions.

To evaluate performance in the sparse regime, we consider synthetic matrices with $1\%$ density (generated using \texttt{sprandn}) and varying dimensions. All sketch-based methods employ 
sparse sign embeddings to ensure fairness. Empirically, the behavior of Blendenpik and Nystr\"om PCG under sparse embeddings is similar to that observed with Gaussian or SRFT sketches, but with reduced sketching cost.

\input{plot_tex/compare_restime_sparse}

\cref{fig:restimes_sparse} demonstrate a clear advantage for \mainalg. \mainalg{} is showing particularly fast convergence for large-scale sparse problems, where it combines low sketching cost and sparsity-aware preconditioner construction. Performance remains consistent across different problem settings. Nystr\"om PCG initially benefits from sparsity when sketching, but forming the Gram matrix destroys sparsity and increases computational cost. Blendenpik relies on dense QR factorizations and therefore does not fully exploit sparse structure. Even in highly overdetermined regimes, where Blendenpik is typically competitive, \mainalg{} reaches the optimal residual significantly faster.

Overall, the experiments in this section demonstrate that \mainalg{} combines adaptive spectral correction, numerical robustness under extreme ill-conditioning, and efficient exploitation of sparsity, making it well-suited for large-scale least-squares problems across diverse regimes.
\yn{YN: the experiments are very good, but it felt the test matrices are contrived. Would be nice to have at least one set of real-world examples.}

\section*{Acknowledgments}
This work was supported by the Additional Funding Programme for Mathematical Sciences, delivered by EPSRC (EP/V521917/1), 
the EPSRC grant EP/Y030990/1, 
and the Heilbronn Institute for Mathematical Research.

\bibliographystyle{siamplain}
\bibliography{references}

\appendix
\input{appendices.tex}

\end{document}

%% file: plot_tex/num_stud_rankest2.tex

\begin{figure}[!htbp]
  \centering
  \begin{subfigure}{0.33\textwidth}
    \centering
    \includegraphics[width=\linewidth,height=0.28\textheight,keepaspectratio]
      {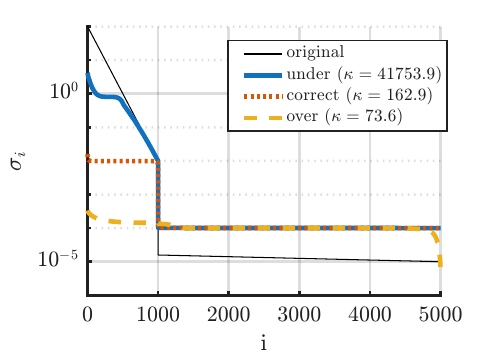}
    \caption{$\sigma_i(\mtx{AP}^{-1})$\\sharp decay}
  \end{subfigure}\hfill
  \begin{subfigure}{0.33\textwidth}
    \centering
    \includegraphics[width=\linewidth,height=0.28\textheight,keepaspectratio]
      {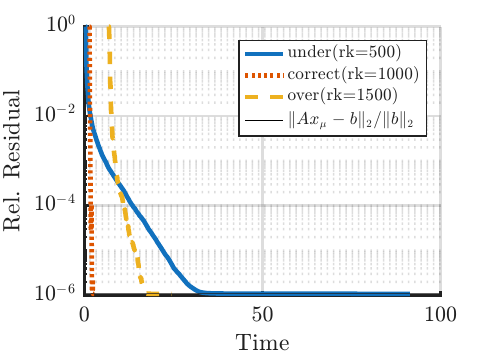}
    \caption{consistent-x\\sharp decay}
  \end{subfigure}\hfill
  \begin{subfigure}{0.33\textwidth}
    \centering
    \includegraphics[width=\linewidth,height=0.28\textheight,keepaspectratio]
      {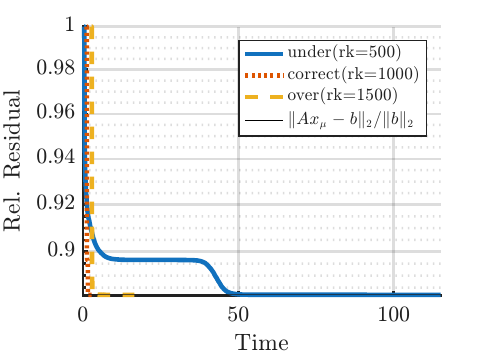}
    \caption{consistent-b\\sharp decay}
  \end{subfigure}\hfill

  \begin{subfigure}{0.33\textwidth}
    \centering
    \includegraphics[width=\linewidth,height=0.28\textheight,keepaspectratio]
      {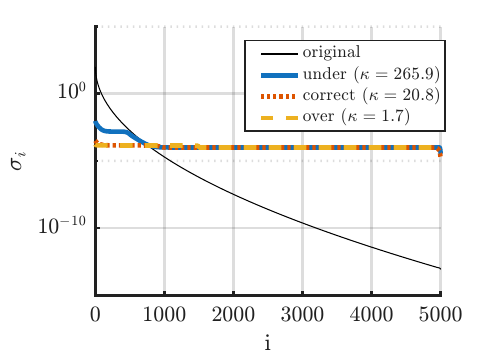}
    \caption{$\sigma_i(\mtx{AP}^{-1})$\\smooth decay}
  \end{subfigure}\hfill
  \begin{subfigure}{0.33\textwidth}
    \centering
    \includegraphics[width=\linewidth,height=0.28\textheight,keepaspectratio]
      {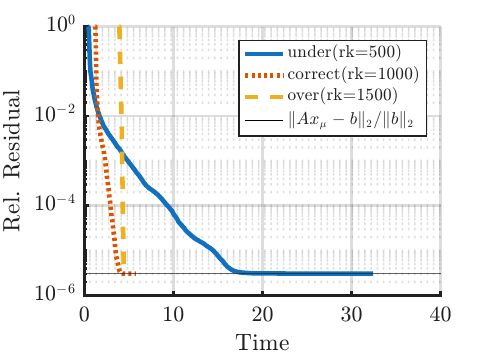}
    \caption{consistent-x\\smooth decay}
  \end{subfigure}\hfill
  \begin{subfigure}{0.33\textwidth}
    \centering
    \includegraphics[width=\linewidth,height=0.28\textheight,keepaspectratio]
      {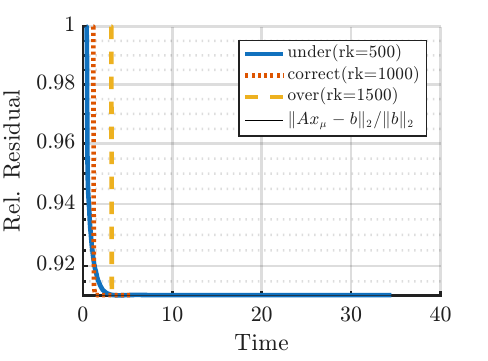}
    \caption{consistent-b\\smooth decay}
  \end{subfigure}\hfill

  \caption{Relative residual convergences for different target ranks using fully-fixed variant, where the target rank is determined by the block size to be 10\%, 20\%, 30\% of the dimension. Subfigures~(a),~(b), and~(c) are on the same dense $6000 \times 5000$ coherent matrix with sharp spectral decay, and subfigures~(d),~(e), and~(f) are on the same dense $6000 \times 5000$ coherent matrix, but with smooth spectral decay. The right-hand side vector $\vct{b}$ differs across problem settings: consistent-x (subfigures~(b),~(e)), consistent-b (subfigures~(c),~(f)).}
  \label{fig:restimes_rankest}
\end{figure}

%% file: plot_tex/num_stud_singleshot.tex
\begin{figure}[!htbp]
  \centering

  \begin{subfigure}{0.25\textwidth}
    \centering
    \includegraphics[width=\linewidth,height=0.28\textheight,keepaspectratio]
      {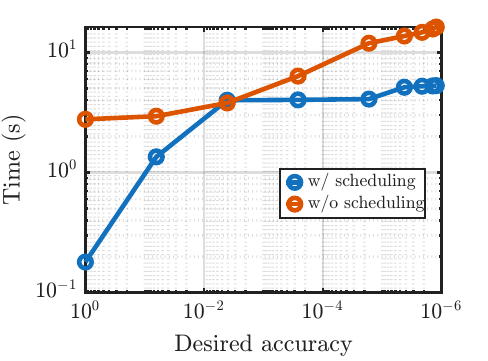}
    \caption{consistent-x\\sharp decay}
  \end{subfigure}\hfill
  \begin{subfigure}{0.25\textwidth}
    \centering
    \includegraphics[width=\linewidth,height=0.28\textheight,keepaspectratio]
      {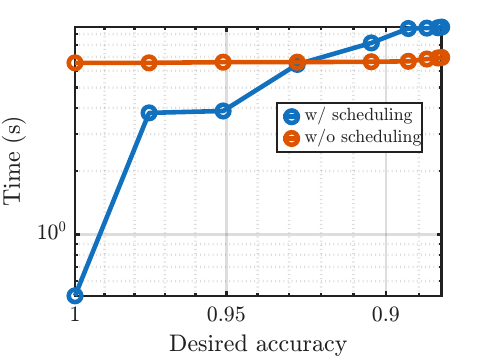}
    \caption{consistent-b\\sharp decay}
  \end{subfigure}\hfill
  \begin{subfigure}{0.25\textwidth}
    \centering
    \includegraphics[width=\linewidth,height=0.28\textheight,keepaspectratio]
      {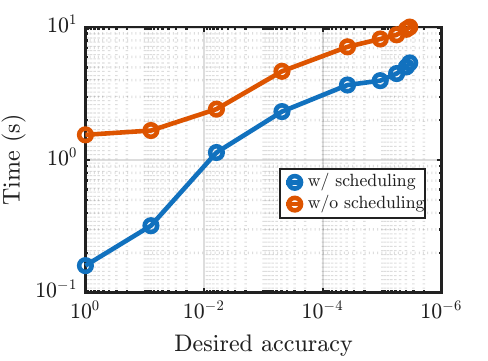}
    \caption{consistent-x\\smooth decay}
  \end{subfigure}\hfill
  \begin{subfigure}{0.25\textwidth}
    \centering
    \includegraphics[width=\linewidth,height=0.28\textheight,keepaspectratio]
      {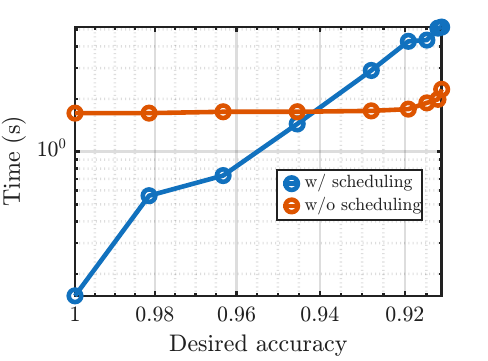}
    \caption{consistent-b\\smooth decay}
  \end{subfigure}\hfill
  
  \caption{Time to reach specific accuracy. The slower convergence of the single-shot variant in subfigures~(a) and~(c) is attributable to problem-dependent effects discussed in Section~\ref{sec:overpreconditioning}. While the single-shot variant can achieve comparable or even faster total time-to-solution in favourable cases, this comes at the expense of substantial upfront preconditioning. In contrast, the proposed fully adaptive algorithm is more robust to problem structure, consistently achieves earlier residual reduction, and avoids excessive setup costs, enabling it to reach moderate accuracy levels significantly faster across a wide range of settings.}
  \label{fig:acctimes8_adaptiveness2}
\end{figure}

\begin{figure}[!htbp]
  \centering

  \begin{subfigure}{0.25\textwidth}
    \centering
    \includegraphics[width=\linewidth,height=0.28\textheight,keepaspectratio]
      {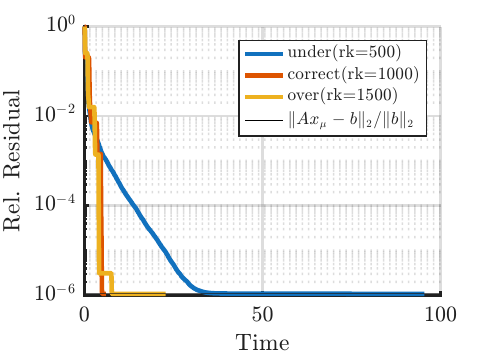}
    \caption{consistent-x\\sharp decay}
  \end{subfigure}\hfill
  \begin{subfigure}{0.25\textwidth}
    \centering
    \includegraphics[width=\linewidth,height=0.28\textheight,keepaspectratio]
      {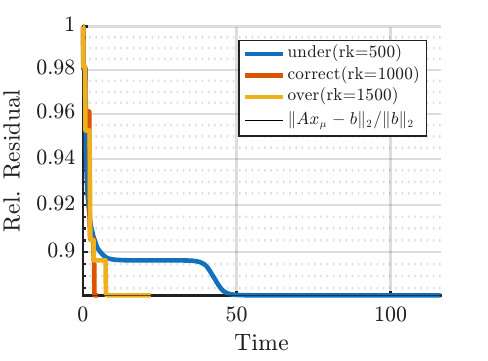}
    \caption{consistent-b\\sharp decay}
  \end{subfigure}\hfill
  \begin{subfigure}{0.25\textwidth}
    \centering
    \includegraphics[width=\linewidth,height=0.28\textheight,keepaspectratio]
      {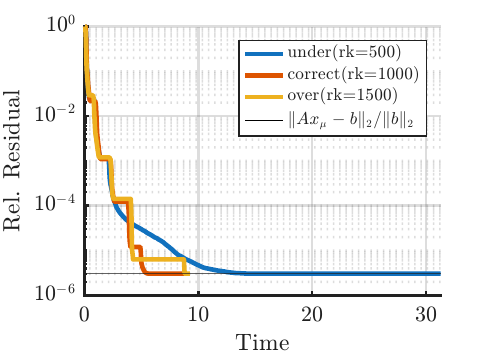}
    \caption{consistent-x\\smooth decay}
  \end{subfigure}\hfill
  \begin{subfigure}{0.25\textwidth}
    \centering
    \includegraphics[width=\linewidth,height=0.28\textheight,keepaspectratio]
      {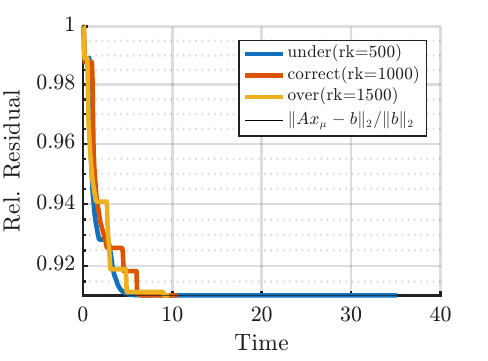}
    \caption{consistent-b\\smooth decay}
  \end{subfigure}\hfill
  
  \caption{Relative residual vs. wall-clock time. Fixing the block size at $n/50$ and varying the number of outer iterations (5, 10, 15), we reach the same target ranks as in Figure~\ref{fig:acctimes8_adaptiveness2}, but through pregressive preconditioner refinement. }
  \label{fig:restimes_rankestfullyadaptive}
\end{figure}

%% file: plot_tex/num_stud_svdfree.tex
\begin{figure}[!htbp]
  \centering

  \begin{subfigure}{0.33\textwidth}
    \centering
    \includegraphics[width=\linewidth,height=0.28\textheight,keepaspectratio]
      {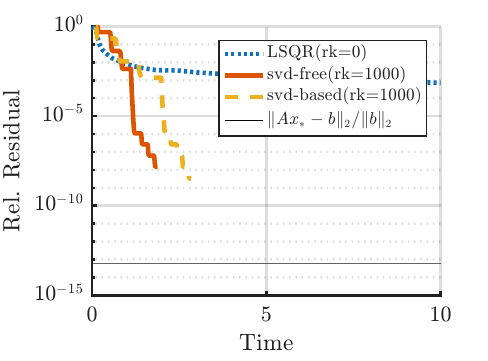}
    \caption{sharp decay\\no regularization\\$\kappa=10^7$}
  \end{subfigure}\hfill
  \begin{subfigure}{0.33\textwidth}
    \centering
    \includegraphics[width=\linewidth,height=0.28\textheight,keepaspectratio]
      {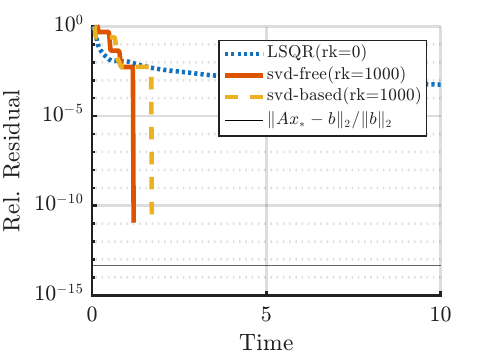}
    \caption{sharp decay\\no regularization\\$\kappa=10^{15}$}
  \end{subfigure}\hfill
  \begin{subfigure}{0.33\textwidth}
    \centering
    \includegraphics[width=\linewidth,height=0.28\textheight,keepaspectratio]
      {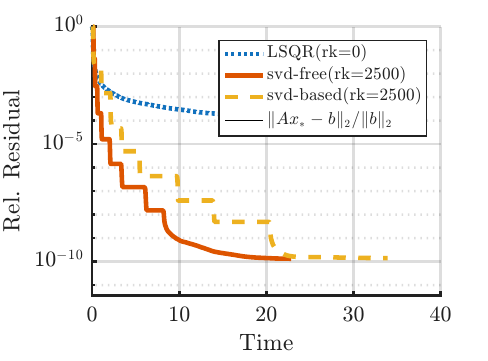}
    \caption{smooth decay\\no regularization\\$\kappa=10^{15}$}
  \end{subfigure}\hfill
  \begin{subfigure}{0.33\textwidth}
    \centering
    \includegraphics[width=\linewidth,height=0.28\textheight,keepaspectratio]
      {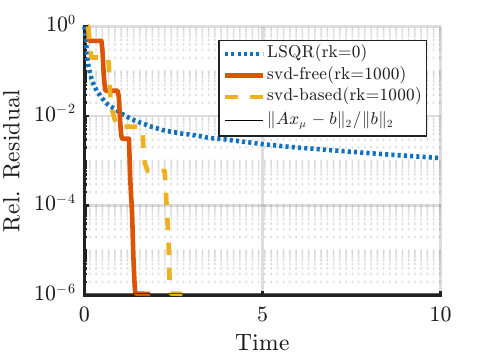}
    \caption{sharp decay\\$\mu = 10^{-4}$\\$\kappa=10^7$}
  \end{subfigure}\hfill
  \begin{subfigure}{0.33\textwidth}
    \centering
    \includegraphics[width=\linewidth,height=0.28\textheight,keepaspectratio]
      {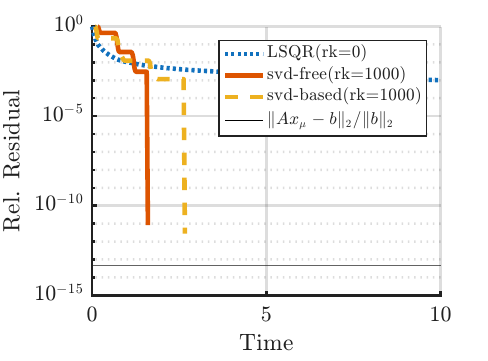}
    \caption{sharp decay\\$\mu = 10^{-10}$\\$\kappa=10^{15}$}
  \end{subfigure}\hfill
  \begin{subfigure}{0.33\textwidth}
    \centering
    \includegraphics[width=\linewidth,height=0.28\textheight,keepaspectratio]
      {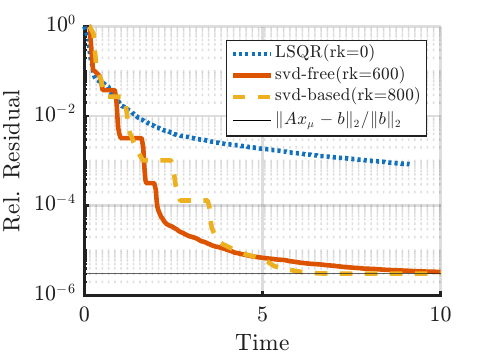}
    \caption{smooth decay\\$\mu = 10^{-4}$\\$\kappa=10^{15}$}
  \end{subfigure}\hfill
  
  \caption{Relative residual vs. wall-clock time. All problems are consistent-x problems with different spectral decays and prescribed regularization parameter $\mu$. For unregularized problems, we are controlling $\varepsilon_\mathrm{cur}$ so that the preconditioning remains low-ranked. (a), (b), and (c) are on unregularized problems. (d), (e), and (f) are on regularized problems.}
  \label{fig:restimes_svdfree}
\end{figure}

%% file: plot_tex/compare_restime.tex
\begin{figure}[!htbp]
  \centering
  \begin{subfigure}{0.33\textwidth}
    \centering
    \includegraphics[width=\linewidth,height=0.28\textheight,keepaspectratio]
        {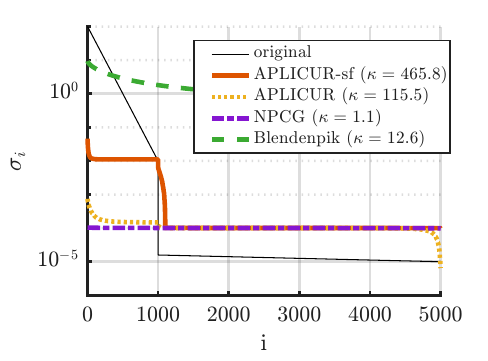}
    \caption{Spectrum of $\mtx{A_{\mu}P_\mu}^{-1}$\\(default problem)}
  \end{subfigure}\hfill
  \begin{subfigure}{0.33\textwidth}
    \centering
    \includegraphics[width=\linewidth,height=0.28\textheight,keepaspectratio]
      {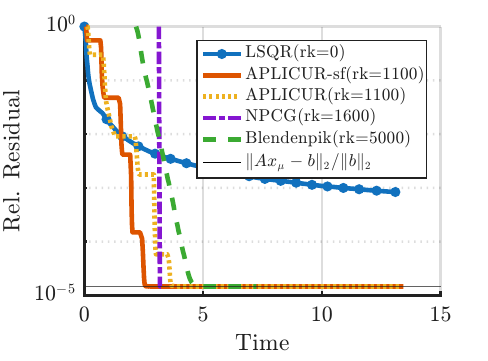}
    \caption{Default\\problem setting}
  \end{subfigure}\hfill
  \begin{subfigure}{0.33\textwidth}
    \centering
    \includegraphics[width=\linewidth,height=0.28\textheight,keepaspectratio]
      {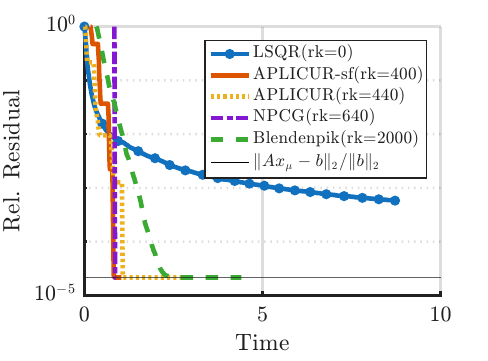}
    \caption{Highly\\overdetermined setting}
  \end{subfigure}\hfill

  \begin{subfigure}{0.33\textwidth}
    \centering
    \includegraphics[width=\linewidth,height=0.28\textheight,keepaspectratio]
      {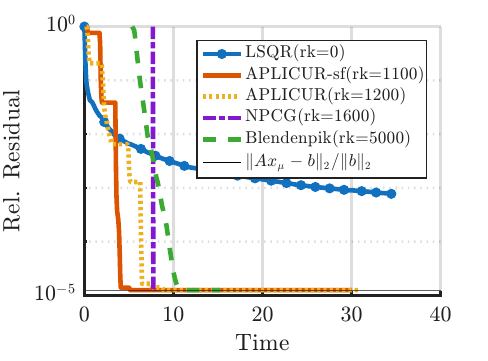}
    \caption{Coherent matrix}
  \end{subfigure}\hfill
  \begin{subfigure}{0.33\textwidth}
    \centering
    \includegraphics[width=\linewidth, height=0.28\textheight, keepaspectratio]
      {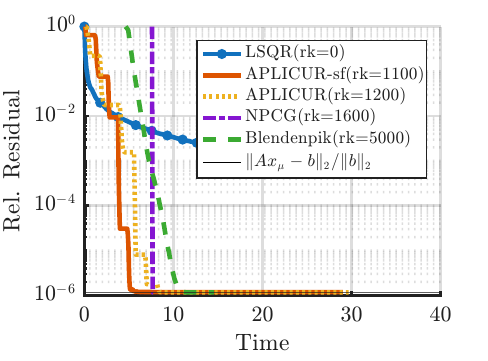}
    \caption{Consistent system}
  \end{subfigure}\hfill
  \begin{subfigure}{0.33\textwidth}
    \centering
    \includegraphics[width=\linewidth, height=0.28\textheight, keepaspectratio]
      {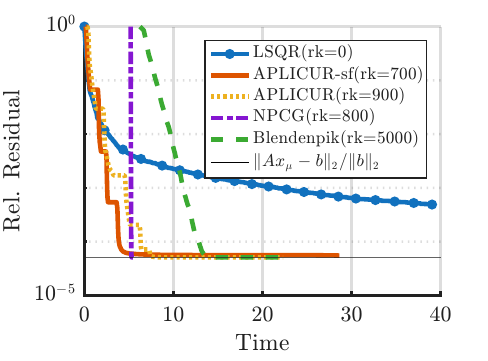}
    \caption{Smooth spectral decay}
  \end{subfigure}\hfill

  \caption{Relative residual versus wall-clock time for LSQR, Nyström PCG, Blendenpik, and \mainalg. (a, b) Default problem: moderately overdetermined ($6000\times5000$), sharp spectral decay ($\kappa=10^7$), incoherent and inconsistent ($10^{-2}$ noise), with regularization $\mu=10^{-4}$. (c–f) Experiments varying one property at a time: (c) highly overdetermined ($15000\times2000$); (d) coherent; (e) consistent (noise-free); (f) smooth spectral decay with $\kappa=10^{15}$.}
  \label{fig:restimes_varioussettings}
\end{figure}

%% file: plot_tex/compare_restime_ill.tex
\begin{figure}[!htbp]
  \centering
  \begin{subfigure}{0.48\textwidth}
    \centering
    \includegraphics[width=\linewidth,height=0.28\textheight,keepaspectratio]
      {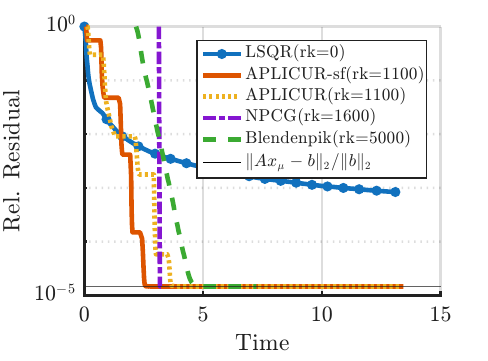}
    \caption{$\;\kappa = 10^7, \mu = 10^{-4}\;$}
  \end{subfigure}\hfill
  \begin{subfigure}{0.48\textwidth}
    \centering
    \includegraphics[width=\linewidth,height=0.28\textheight,keepaspectratio]
      {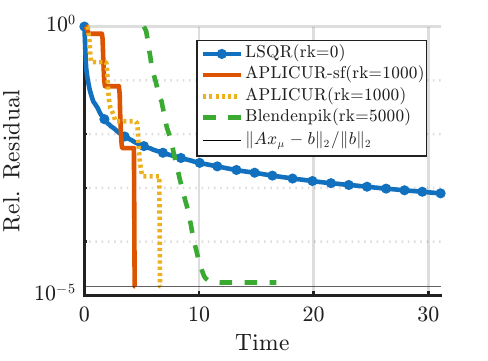} 
    \caption{$\kappa = 10^{15}, \mu = 10^{-8}$}
  \end{subfigure}\hfill
  
  \caption{Relative residual versus wall-clock time for LSQR, Nyström PCG, Blendenpik, and \mainalg on a consistent-$x$ problem with $10^{-2}$ noise, using an incoherent $6000 \times 5000$ matrix with sharp spectral decay.}
  \label{fig:restimes_ill}
\end{figure}

%% file: plot_tex/compare_restime_reg.tex
\begin{figure}[!htbp]
  \centering

    \begin{subfigure}{0.24\textwidth}
    \centering
    \includegraphics[width=\linewidth,height=0.28\textheight,keepaspectratio]
      {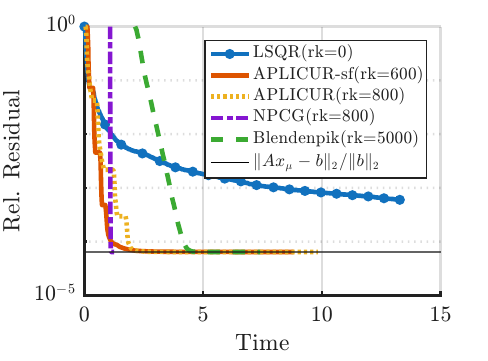}
    \caption{$\mu = 10^{-4}$}
  \end{subfigure}\hfill
  \begin{subfigure}{0.24\textwidth}
    \centering
    \includegraphics[width=\linewidth,height=0.28\textheight,keepaspectratio]
      {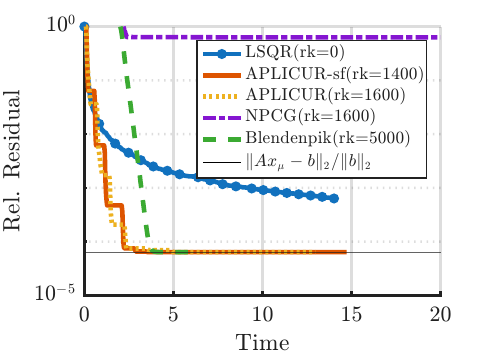}
    \caption{$\mu = 10^{-6}$}
  \end{subfigure}\hfill
  \begin{subfigure}{0.24\textwidth}
    \centering
    \includegraphics[width=\linewidth,height=0.28\textheight,keepaspectratio]
      {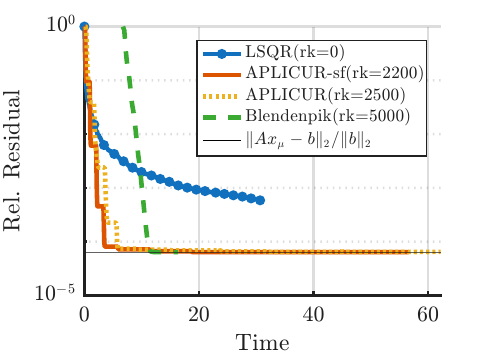}
    \caption{$\mu = 10^{-8}$}
  \end{subfigure}\hfill
  \begin{subfigure}{0.24\textwidth}
    \centering
    \includegraphics[width=\linewidth,height=0.28\textheight,keepaspectratio]
      {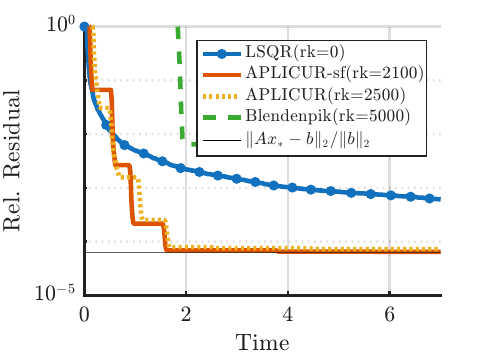}
    \caption{$\mu = 0$}
  \end{subfigure}\hfill
  
  \caption{Relative residual versus wall-clock time for LSQR, Nyström PCG, Blendenpik, and \mainalg on a consistent-$x$ problem with $10^{-2}$ noise, using an incoherent $6000 \times 5000$ matrices with smooth spectral decay ($\kappa = 10^{15}$), under varying levels of regularization.}
  \label{fig:restimes_reg}
\end{figure}

%% file: plot_tex/compare_restime_sparse.tex
\begin{figure}[!htbp]
  \centering

  \begin{subfigure}{0.24\textwidth}
    \centering
    \includegraphics[width=\linewidth,height=0.28\textheight,keepaspectratio]
      {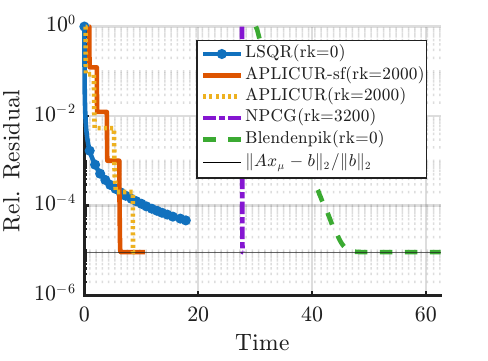}
    \caption{Moderately overdetermined,\\sharp decay}
  \end{subfigure}\hfill
  \begin{subfigure}{0.24\textwidth}
    \centering
    \includegraphics[width=\linewidth,height=0.28\textheight,keepaspectratio]
      {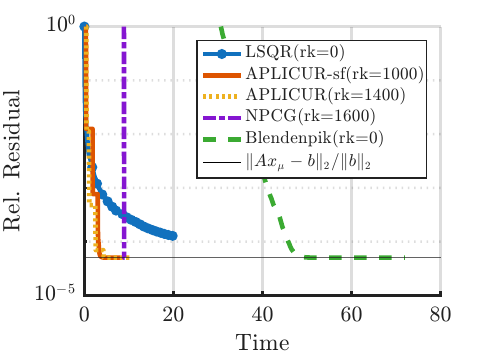}
    \caption{Moderately overdetermined,\\smooth decay}
  \end{subfigure}\hfill
    \begin{subfigure}{0.24\textwidth}
    \centering
    \includegraphics[width=\linewidth,height=0.28\textheight,keepaspectratio]
      {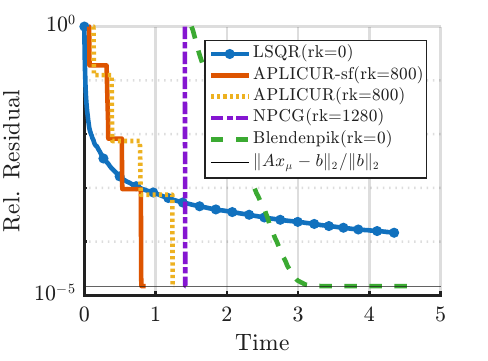}
    \caption{Highly overdetermined,\\sharp decay}
  \end{subfigure}\hfill
  \begin{subfigure}{0.24\textwidth}
    \centering
    \includegraphics[width=\linewidth,height=0.28\textheight,keepaspectratio]
      {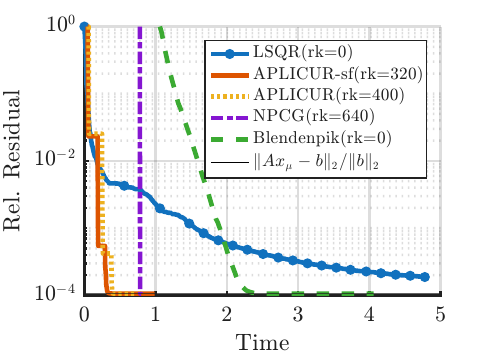}
    \caption{Highly overdetermined,\\smooth decay}
  \end{subfigure}\hfill
  
  \caption{Relative residual versus wall-clock time for LSQR, Nyström PCG, Blendenpik, and \mainalg on moderately overdetermined ($12000\times 10000$) and highly overdetermined ($30000 \times 4000$) matrices.}
  \label{fig:restimes_sparse}
\end{figure}

%% file: appendices.tex
\section{Proofs}
This appendix provides proofs of the theoretical results supporting the main text

\subsection{Proof of Theorem~\ref{thm:conditionnumberbound}}\label{proof:conditionnumberbound}

\begin{proof}
We first analyze the idealized system based on $\mtx{\widehat A}_\l$ and then account for the perturbation.
Let $\mtx{\widehat A}_\l = \mtx{\widehat U \widehat \Sigma \widehat V^\top}$ be the SVD of $\mtx{\widehat{A}}$.

\paragraph{Step 1: Idealized system}

Define
\[
\mtx{\widehat A}_{\l,\mu} := \begin{bmatrix} \mtx{\widehat A}_{\l} \\ \mu \mtx{I} \end{bmatrix}.
\]
Let
\[
\mtx{\widehat\Sigma}_{\mathrm{full}} =
\begin{bmatrix}
\mtx{\widehat\Sigma} & \\
& \mtx{0}
\end{bmatrix}, 
\quad
\mtx{\widehat V}_{\mathrm{full}} = [\mtx{\widehat V} \ \mtx{\widehat V}_\perp].
\]

By the construction of the preconditioner,
\begin{align*}
\mtx{P}_{\l,\mu}^{-\top}\mtx{\widehat A}_{\l,\mu}^\top \mtx{\widehat A}_{\l,\mu} \mtx{P}_{\l,\mu}^{-1}
&= \mtx{P}_{\l,\mu}^{-\top}(\mtx{\widehat A}_\l^\top \mtx{\widehat A}_{\l} + \mu^2 \mtx{I})\mtx{P}_{\l,\mu}^{-1} \\
&= \mtx{\widehat V}_{\mathrm{full}}
\begin{bmatrix}
(\widehat\sigma^2+\mu^2) \mtx{I}_\ell & \\
& \mu^2 \mtx{I}_{n-\l}
\end{bmatrix}
\mtx{\widehat V}_{\mathrm{full}}^\top.
\end{align*}

Hence, the singular values of $\mtx{\widehat A}_{\l,\mu} \mtx{P}_{\l,\mu}^{-1}$ are
\[
\sqrt{\widehat\sigma^2 + \mu^2} \quad (\ell \text{ times}), 
\qquad
\mu \quad (n-\ell \text{ times}).
\]

In other words, $\mtx{P}_{\l,\mu}^{-1}$ rescales only the $\mtx{\widehat V}$-subspace and acts as an isometry on its orthogonal complement.

\paragraph{Step 2: Perturbation}

Write
\[
\mtx{A}_\mu \mtx{P}_{\l,\mu}^{-1}
= \mtx{\widehat A}_{\l,\mu} \mtx{P}_{\l,\mu}^{-1}
+ \underbrace{\begin{bmatrix} \mtx{E} \\ \mtx{0} \end{bmatrix} \mtx{P}_{\l,\mu}^{-1}}_{:=\mtx{\Delta}}.
\]

Using $\svhat \le \sqrt{\widehat\sigma_\ell^2 + \mu^2}$, we have $\|\mtx{P}_{\l,\mu}^{-1}\|_2 \le 1$, and therefore
$\|\mtx{\Delta}\|_2 \le \|\mtx{E}\|_2$.
Applying Weyl's inequality gives
\begin{align*}
\sigma_{\max}(\mtx{A}_\mu \mtx{P}_{\l,\mu}^{-1})
&\le \sqrt{\widehat\sigma^2 + \mu^2} + \|\mtx{E}\|_2,\\
\sigma_{\min}(\mtx{A}_\mu \mtx{P}_{\l,\mu}^{-1})
&\ge \mu - \|\mtx{E}\|_2.
\end{align*}

\paragraph{Step 3: Condition number}

If $\|\mtx{E}\|_2 < \mu$, then
\[
\kappa(\mtx{A}_\mu \mtx{P}_{\l,\mu}^{-1})
\le \frac{\sqrt{\widehat\sigma^2 + \mu^2} + \|\mtx{E}\|_2}{\mu - \|\mtx{E}\|_2}.
\]
\end{proof}

\subsection{Proof of Theorem~\ref{thm:cvgbound_disjoint}}\label{app:twointervalbound}

The classical Chebyshev bound depends only on the global condition number $\kappa(\mathbf{A})$ and is therefore often pessimistic for systems with clustered spectra.

Assume that the dominant singular values lie in a narrow interval:
\[
\sigma_1^2, \ldots, \sigma_\ell^2 \in [\svhat^2 - \varepsilon,\ \svhat^2 + \varepsilon],
\]
while the remaining singular values lie below $\sigma_{\ell+1}$.
Thus, the spectrum is contained in two disjoint intervals:
\[
[\sigma_n^2, \sigma_{\ell+1}^2]
\cup
[\sigma_\ell^2, \sigma_1^2].
\]

To exploit this structure, we bound the LSQR residual via polynomial approximation over the union
\[
\mathcal{S}
:=
[a,b] \cup [c,d] \supseteq [\sigma_n^2, \sigma_{\ell+1}^2]
\cup
[\sigma_\ell^2, \sigma_1^2],
\]
where both intervals of $\mathcal{S}$ have equal length $w$.

\begin{lemma}[\cref{thm:cvgbound_disjoint}]
Under the above assumptions, the LSQR residual after the $k$th iterate $\vct{x}_k$ satisfies
\[
\|\mtx{U}^\top(\mtx{A}\vct{x}_k - \vct{b})\|_2
\le
\left(\frac{\sqrt{C}-1}{\sqrt{C}+1}\right)^{\lfloor k/2 \rfloor},
\]
where $\mtx{U}$ is the left singular vector matrix of $\mtx{A}$ and
\[
C
=
\kappa_{\ell+1}^2 + \frac{w}{\sigma_n^2}.
\]
\end{lemma}

\begin{proof}[Proof of \cref{thm:cvgbound_disjoint}]
Define $w_1 := \sigma_1^2 - \sigma_\ell^2$ and $w_2 := \sigma_{\ell+1}^2 - \sigma_n^2$, and let $w := \max\{w_1, w_2\}$. Then set
\begin{align}
    a:= \sigma_n^2, \quad b:= \sigma_n^2 + w > \sigma_{\ell+1}^2,
    \quad c:= \sigma_\ell^2, \quad d:= \sigma_\ell^2 + w > \sigma_1^2.
\end{align}
This constructs a union of two disjoint intervals of equal length $w$:
\[
\mathcal{S} := [a, b] \cup [c, d] \supseteq [\sigma_n^2, \sigma_{\ell+1}^2] \cup [\sigma_\ell^2, \sigma_1^2].
\]

We introduce the quadratic polynomial
\[
q_2(\lambda) = 1 + \frac{2(\lambda - b)(\lambda - c)}{ad - bc},
\]
which maps both intervals of $\mathcal{S}$ to $[-1,1]$ monotonically (this is possible because the intervals have equal length~\cite{ashby1988polynomial}). Then consider the polynomial
\begin{align}
    p_k(z) = \frac{T_{\lfloor k/2 \rfloor}(q_2(z))}{T_{\lfloor k/2 \rfloor}(q_2(0))}.
\end{align}
By the boundedness property of the ordinary Chebyshev polynomial on $[-1,1]$, we obtain
\begin{align}
    \max_{z \in \mathcal{S}} |p_k(z)| \leq \frac{1}{\left|T_{\lfloor k/2 \rfloor}\left(\frac{bc + ad}{bc - ad}\right)\right|}
    \leq \left(\frac{\sqrt{\frac{bc}{ad}}-1}{\sqrt{\frac{bc}{ad}}+1}\right)^{\lfloor k/2 \rfloor},
\end{align}
where the equal-length property of the intervals ensures $bc > ad$.

To bound $\frac{bc}{ad}$, define $\eta := |w_1 - w_2| < \max\{w_1, w_2\}$ and derive
\begin{align*}
    \frac{bc}{ad} 
    &\leq \frac{\sigma_{\ell+1}^2 + \eta}{\sigma_n^2}\cdot\frac{\sigma_\ell^2}{\sigma_1^2} 
    \leq \frac{\sigma_{\ell+1}^2 + \eta}{\sigma_n^2} 
    < \kappa_{\ell+1}^2 + \frac{\max\{w_1, w_2\}}{\sigma_n^2}\\
    &\leq \kappa_{\ell+1}^2 + \max\left\{\frac{2\varepsilon}{\sigma_n^2}, \kappa_{\ell+1}^2 - 1\right\}.
\end{align*}
\end{proof}

\begin{remark}
The bound depends only on $\kappa_{\ell+1}$ and the cluster widths, rather than on the global condition number $\kappa(\mtx A)$. Hence, large spectral gaps do not degrade convergence once clustering is achieved.
\end{remark}


\section{Implementation Details}
This appendix provides additional implementation details, including an overview of the CUR approximation method and technical aspects of efficient implementation, such as factorization strategies and practical considerations.

\subsection{CUR approximation details}\label{app:cur_details}

There exists a wide variety of CUR approximations, which differ in three main aspects: (1) row and column selection, (2) construction of the core matrix $\mtx{U}$, and (3) the choice of target rank $\ell$.

\subsubsection{Row and column selection}

Selecting appropriate rows and columns is crucial to the quality of the CUR approximation. Existing methods are broadly classified into two categories. The first consists of sampling-based approaches, which draw indices according to specific probability distributions, including leverage score sampling \cite{drineas2008relative, mahoney2009cur}, uniform sampling \cite{chiu2013sublinear}, volume sampling \cite{deshpande2006matrix, deshpande2010efficient, cortinovis2020low}, determinantal point process (DPP) sampling \cite{derezinski2021determinantal}, K-means++ \cite{oglic2017nystrom}, and BSS sampling \cite{batson2014twice, boutsidis2014near}. The second category uses pivoting-based approaches, where strategies such as column-pivoted QR (CPQR) \cite{golub2013matrix}, LU with partial or complete pivoting \cite{trefethen2022numerical, dong2023simpler}, and strong rank-revealing factorisations \cite{hong1992rank, gu1996efficient, pan2000existence} identify informative rows and columns via the permutation vectors produced during factorization.

For large matrices, directly selecting informative indices from $\mtx{A}$ can be costly, so randomized techniques construct a smaller ``sketch'' $\mtx{X}$ that preserves sufficient information for selection \cite{halko2011finding, drineas2008relative, voronin2017efficient, duersch2020randomized, dong2023simpler}. Given a matrix $\mtx{A} \in \mathbb{R}^{m \times n}$, a sketch is formed by multiplying $\mtx{A}$ by a random embedding matrix $\mtx{S}$, yielding $\mtx{X} = \mtx{S}\mtx{A}$ (or $\mtx{X} = \mtx{A}\mtx{S}$ depending on whether rows or columns are being sketched). Several sketching techniques are commonly used, including Gaussian sketches \cite{halko2011finding, indyk1998approximate, martinsson2020randomized, woodruff2014sketching}, subsampled randomized trigonometric transforms (SRTT) \cite{ailon2009fast, halko2011finding, martinsson2020randomized, woolfe2008fast, tropp2011improved, rokhlin2008fast}, hashing embeddings \cite{cartis2021hashing}, and sparse sign embeddings \cite{martinsson2020randomized, nelson2013osnap, kane2014sparser, woodruff2014sketching, clarkson2017low, meng2013low}. Sketching is known to incur only minor losses in accuracy \cite{halko2011finding}.

\subsubsection{Choice of the core matrix}

There are two main choices for the core matrix $\mtx{U}$, namely CUR with best approximation (CUR-BA) and CUR with cross approximation (CUR-CA). CUR-BA chooses $\mtx{U}$ to minimize $\|\mtx{A} - \mtx{C}\mtx{U}\mtx{R}\|_F$, namely
\[
\mtx{U} = \mtx{C}^\dagger \mtx{A} \mtx{R}^\dagger,
\]
and theoretical bounds showing that the resulting approximation error is at most a factor $\sqrt{2\ell+2}$ larger than the best rank-$\ell$ approximation error can be found in \cite{sorensen2016deim, dong2023simpler, osinsky2025close}.

CUR-CA instead chooses $\mtx{U}$ to be the pseudoinverse of the intersection of the selected columns and rows, namely $\mtx{U} = \mtx{A}(\vct{I}, \vct{J})^\dagger$. This construction requires evaluating $\mtx{A}$ only along the cross of $\ell$ rows and $\ell$ columns, which is particularly advantageous when full access to the matrix is not readily available. However, CUR-CA has the drawback that $\mtx{A}(\vct{I},\vct{J})$ can be nearly singular, leading to unfavorable approximation error due to numerical instability in computing the pseudoinverse \cite{dong2023simpler, martinsson2020randomized, sorensen2016deim}. For this reason, variants such as the $\varepsilon$-pseudoinverse \cite{chiu2013sublinear} or oversampling \cite{park2025accuracy} have been proposed to avoid numerical instability.

Despite these concerns, we adopt CUR-CA due to its low computational cost and its suitability for constructing lightweight preconditioners. In our setting, stability is maintained through oversampled sketching, which mitigates ill-conditioning in $\mtx{A}(\vct{I}, \vct{J})$. The full procedure is given in \cref{alg:cur}. The total computational cost is $\mathcal{O}((m+n+\ell)\ell^2)$. Specifically, forming $\mtx{C}$ via randomized LUPP on $\mtx{Y}^\top \in \F^{n \times d}$ costs $\mathcal{O}(nd^2)$, while forming $\mtx{R}$ via LUPP on $\mtx{C} \in \F^{m \times \ell}$ costs $\mathcal{O}(m\ell^2)$.

\begin{algorithm}
    \caption{Sketched CUR with cross approximation using LUPP \cite{dong2023simpler}}\label{alg:cur}
    \begin{algorithmic}[1]
        \Require{$\mtx{A} \in \F^{m \times n}$, target rank $\ell$}
        \Ensure Row indices $\vct{I}$, column indices $\vct{J}$
        \Statex
        \Function{CURApproximation}{$\mtx{A}, \ell$}
            \State $\mtx{S} \gets \texttt{sparsesign}(\lfloor 1.1\ell \rfloor, m, 8)$ \Comment{sparse sign embedding with $\xi = 8$}
            \State $[\sim, \sim, \vct{P}_{\mathrm{col}}] \gets \texttt{lu}((\mtx{S}\mtx{A})^\top, \texttt{`vector'})$
            \Comment{LUPP of $(\mtx{S}\mtx{A})^\top$}
            \State $\vct{J} \gets \vct{P}_{\mathrm{col}}(1:\ell)$ \Comment{select $\ell$ columns}
            \State $\mtx{C} \gets \mtx{A}(:,\vct{J})$

            \State $[\sim, \sim, \vct{P}_{\mathrm{row}}] \gets \texttt{lu}(\mtx{C}, \texttt{`vector'})$
            \Comment{LUPP of $\mtx{C}$}
            \State $\vct{I} \gets \vct{P}_{\mathrm{row}}(1:\ell)$
            \Comment{select $\ell$ rows}
            \State $\mtx{R} \gets \mtx{A}(\vct{I}, :)$
            \State $\mtx{U} \gets \mtx{A}(\vct{I}, \vct{J})^\dagger$ 
            \Comment{default QR; LU is more efficient}
            
            \State \Return $\mtx{C}, \mtx{U}, \mtx{R}$
        \EndFunction
    \end{algorithmic}
\end{algorithm}

\begin{theorem}[Theorem 2 in \cite{osinsky2025close}]\label{thm:curosinsky}
    Let $\mtx{A}, \mtx{Z} \in \mathbb{R}^{m \times n}$ with $\mathrm{rank}(\mtx{Z}) = \ell$. Then, in $\mathcal{O}(mn\ell)$ operations, it is possible to find rows $\mtx{R} \in \mathbb{R}^{\ell \times n}$ and columns $\mtx{C} \in \mathbb{R}^{m \times \ell}$ of the matrix $\mtx{A}$ such that 
    \begin{align*}
        \|\mtx{A} - \mtx{C}\mtx{C}^\dagger \mtx{A} \mtx{R}^\dagger \mtx{R}\|_F &\leq \sqrt{2+2\ell}\,\|\mtx{A} - \mtx{Z}\|_F,\\
        \|\mtx{A} - \mtx{C}\mtx{C}^\dagger \mtx{A} \mtx{R}^\dagger \mtx{R}\|_2 &\leq \sqrt{2+2\ell(\min(m,n)-\ell)}\,\|\mtx{A} - \mtx{Z}\|_2.
    \end{align*}
    In the same number of operations, it is also possible to select the same rows $\mtx{R} \in \mathbb{R}^{\ell \times n}$ but possibly different columns $\mtx{C} \in \mathbb{R}^{m \times \ell}$, such that
    \begin{align*}
        \|\mtx{A} - \mtx{C}\mtx{M}^\dagger \mtx{R}\|_F &\leq (\ell+1)\,\|\mtx{A}-\mtx{Z}\|_2,\\
        \|\mtx{A} - \mtx{C}\mtx{M}^\dagger \mtx{R}\|_2 &\leq \sqrt{1+\ell(\ell+2)(\min(m,n)-\ell)}\,\|\mtx{A}-\mtx{Z}\|_2,
    \end{align*}
    where $\mtx{M} \in \mathbb{R}^{\ell \times \ell}$ is the submatrix at the intersection of the selected rows and columns.
\end{theorem}

\begin{corollary}[\cite{osinsky2025close,zamarashkin2018existence}]
Using $\mtx{Z}=\mtx{A}_\ell$, where $\mtx{A}_\ell$ is the rank-$\ell$ truncated SVD of $\mtx{A}$, we obtain
\begin{align*}
    \|\mtx{A} - \mtx{C}\mtx{M}^\dagger \mtx{R}\|_F \leq (\ell+1)\|\mtx{A} - \mtx{A}_\ell\|_F.
\end{align*}
\end{corollary}

Standard CUR approximation often monitors the Frobenius norm to reduce aggregate approximation error, due to its unbiasedness and lower variance~\cite{horesh2025variance}.

\subsubsection{Choice of target rank}

Choosing the target rank $\ell$ is crucial for both the quality and efficiency of the preconditioner. If the target rank is too large, it can cause numerical instability in LUPP or pseudoinverse computations and increase computational cost, undermining the purpose of low-rank preconditioning. If it is too small, the best rank-$\ell$ approximation error, $\sigma_{\ell+1}(\mtx{A})$, may be large, preventing accurate approximation. Many works set the target rank of a low-rank approximation proportional to the effective dimension \cite{frangella2023randomized} or statistical rank of $\mtx{A}$, but in practice these quantities are typically unknown. We therefore adopt two adaptive strategies for selecting $\ell$ in the CUR approximation.

The fast rank estimation method of \cite{meier2024fast} estimates the numerical rank by repeatedly sketching the matrix from both sides and performing an SVD on the resulting compressed matrix. In particular, Gaussian or SRTT sketches are used to approximate the singular values, from which the rank is inferred (see Algorithm~\ref{alg:cur_rankest}). The total complexity is $\mathcal{O}((m+\log n)n\ell)=\mathcal{O}(mn\ell)$, where the Gaussian sketch costs $\mathcal{O}(mn\ell)$, the SRTT sketch costs $\mathcal{O}(n\ell\log n + \ell^3)$, and the SVD costs $\mathcal{O}(\ell^3)$. In practice, these sketches can be appended rather than recomputed at each iteration, and the final sketch can be reused for the subsequent randomized CUR approximation, so the overall cost remains modest. However, when the estimated rank $\ell$ is large, the SVD of the sketched matrix of size $\ell \times 2\ell$ can become a computational bottleneck.

As an alternative, Pritchard \textit{et al.}~\cite{pritchard2025fast} propose \emph{IterativeCUR}, a fast rank-adaptive CUR approximation method (outlined in Algorithm~\ref{alg:cur_icur}) that requires only a single sketch $\mtx{S}\mtx{A}$, independent of the target rank. The method incrementally increases the CUR rank by selecting additional column and row indices based on the sketched residual $\mtx{S}\mtx{A} - \mtx{S}\mtx{C}\mtx{U}\mtx{R}$. Crucially, the initial sketch $\mtx{S}\mtx{A}$ is reused throughout, and products involving $\mtx{S}\mtx{C}$ are extracted from it, avoiding any additional sketching. The CUR factors are updated iteratively by augmenting the index sets using the procedure described in Algorithm~\ref{alg:cur_augment}. This yields an efficient rank-adaptive scheme in which the approximation is refined progressively while maintaining low computational cost.

\begin{algorithm}
    \caption{Fast rank estimation \cite{meier2024fast}}
    \label{alg:cur_rankest}
    \begin{algorithmic}[1]
        \Require $\mtx{A} \in \F^{m \times n}$; rank tolerance $\varepsilon_{\text{rank}}$; initial rank upper-bound estimate $\ell_0$.
        \Ensure Estimated rank $\ell$.
        \Statex
        \While{true}
        \State $\mtx{S}_1 \gets \texttt{sparsesign}(\ell_0, m, 8)$ \Comment{$\ell_0 \times m$ sparse sign embedding with $\xi = 8$}
        \State $\mtx{X} \gets \mtx{S}_1\mtx{A}$
        \State $\mtx{S}_2 \gets \texttt{SRTT}(2\ell_0, n)$ \Comment{$2\ell_0 \times n$ SRTT embedding}
        \State $[\sigma_1, \ldots, \sigma_{\ell_0}] \gets \texttt{svd}(\mtx{X}\mtx{S}_2^\top)$
        \If{$\exists\, \hat{\ell}\ \text{s.t.}\ \sigma_{\hat{\ell}+1} \le \varepsilon_{\text{rank}} \cdot \sigma_{1}$}
            \State $\ell \gets \hat{\ell}$ \Comment{estimated $\varepsilon_{\text{rank}}$-rank of $\mtx{A}$}
            \State \texttt{break}
        \Else
            \State $\ell_0 \gets 2\ell_0$ \Comment{increase rank upper bound and repeat}
        \EndIf
        \EndWhile
    \end{algorithmic}
\end{algorithm}

\begin{algorithm}
    \caption{IterativeCUR \cite{pritchard2025fast}}
    \label{alg:cur_icur}
    \begin{algorithmic}[1]
        \Require{$\mtx{A} \in \F^{m \times n}$; block size $\ell_0$ (usually in $[5,250]$); error tolerance $\varepsilon_{\text{cur}}$.}
        \Statex
        \State $\vct{I} = \vct{J} = \emptyset,\quad \mtx{C}=\mtx{U}=\mtx{R}=\emptyset$
        \State $\mtx{S} \gets \texttt{sparsesign}(1.1\ell_0, m, 8)$ \Comment{oversampling}
        \State $\mtx{E}^{\mathrm{row}} \gets \mtx{S}\mtx{A}$ \Comment{sketch once and reuse it}
        \Repeat \Comment{rank grows by $\ell_0$ per iteration}
            \State $[\vct{I}, \vct{J}] \gets \texttt{augmentCUR}(\mtx{A}, \mtx{C}, \mtx{U}, \mtx{R}, \mtx{E}, \vct{I}, \vct{J})$ \Comment{use \cref{alg:cur_augment}}
            \State $\mtx{U} \gets \mtx{A}(\vct{I}, \vct{J})^\dagger$
            \Comment{default QR; LU is more efficient}
            \State $\mtx{C} \gets \mtx{A}(:, \vct{J})$, $\mtx{R} \gets \mtx{A}(\vct{I},:)$
            \State $\mtx{E}^{\mathrm{row}} \gets \mtx{S}\mtx{A} - \mtx{S}\mtx{C}\mtx{U}\mtx{R}$
            \Comment{use $\mtx{S}\mtx{C}$ from $\mtx{S}\mtx{A}$}
            \State $\rho \gets \|\mtx{E}^{\mathrm{row}}\|_F / \|\mtx{S}\mtx{A}\|_F$
        \Until{$\rho \leq \varepsilon_{cur}$}
        \Statex
        \State \Return $\mtx{C}, \mtx{U}, \mtx{R}$ \Comment{CUR approximation with rank $\ell_0 \times$(\# iterations)}
    \end{algorithmic}
\end{algorithm}

\begin{algorithm}
    \caption{\texttt{augmentCUR}}
    \label{alg:cur_augment}
    \begin{algorithmic}[1]
        \Require Matrix $\mtx{A} \in \F^{m \times n}$; CUR approximation $\mtx{C},\ \mtx{U},\ \mtx{R}$; row-sketched CUR residual $\mtx{E}^{\mathrm{row}} \in \F^{\ell_0 \times n}$; current row and column index sets $\vct{I}$ and $\vct{J}$.
        \Ensure $\ell = \texttt{size}(\mtx{C},2)$ previous approximation rank
        
        \Statex
        \State $\mtx{E}^{\mathrm{row}}(:,\vct{J})\gets \mtx{0}$ \Comment{avoid reselecting previously chosen indices}
        \State $[\sim, \sim, \vct{P}_{\mathrm{col}}] \gets \texttt{lu}({\mtx{E}^{\mathrm{row}}}^\top, \texttt{`vector'})$ \Comment{default LUPP; CPQR or Osinsky can also be used}
        \State $\vct{J}_{+} \gets \vct{P}_{\mathrm{col}}(1:\ell_0)$
        \State $\vct{J} \gets [\vct{J}, \vct{J}_{+}]$
        \Statex
        \State $\mtx{E}^{\mathrm{col}} \gets \mtx{A}(:,\vct{J}_{+}) - \mtx{CUR}(:, \vct{J}_{+})$
        \State $\mtx{E}^{\mathrm{col}}(\vct{I},:)\gets \mtx{0}$ \Comment{avoid reselecting previously chosen indices}
        \State $[\sim, \sim, \vct{P}_{\mathrm{row}}] \gets \texttt{lu}(\mtx{E}^{\mathrm{col}}, \texttt{`vector'})$ \Comment{default LUPP; CPQR or Osinsky can also be used}
        \State $\vct{I}_{+} \gets \vct{P}_{\mathrm{row}}(1:\ell_0)$
        \State $\vct{I} \gets [\vct{I}, \vct{I}_{+}]$

        \Statex
        \State \Return $\vct{I},\ \vct{J}$
        \Comment{$\ell_0$ additional row and column indices}
    \end{algorithmic}
\end{algorithm}

\subsection{Supplementary algorithms}

\begin{algorithm}
    \caption{\texttt{SVD-free CURPreconditioner (\cref{sec:svdfreepreconditioner})}}
    \label{alg:svdfreepreconditioner}
    \begin{algorithmic}[1]
        \Require{Matrix $\mtx{A}$, index sets $\mtx{I}$ and $\vct{J}$ of size $\l$, regularization parameter $\mu$.}
        \Ensure Preconditioner inverse operator $\mtx{\widetilde{P}}_\l^{-1}$

        \State $\mtx{C} \gets \mtx{A}(:, \vct{J})$
        \State $\mtx{R} \gets \mtx{A}(\vct{I}, :)$
        \Statex
        \State \textbf{Part 1: Factorize $\mtx{C}$ and $\mtx{R}$}
        \State $\mtx{T_C} \gets \texttt{chol}(\mtx{C}^\top \mtx{C})$ \Comment{sketched Cholesky QR for stability}
        \State $[\mtx{Q_R}, \mtx{T_R}] \gets \texttt{qr}(\mtx{R}^\top)$
        \Comment{sketched Cholesky QR for efficiency}
        
        \Statex
        \State \textbf{Part 2: Build the spectral preconditioning operator}
        \State $\mtx{M}^{-1} \gets \mtx{T_R}^{-\top}\mtx{A}(\vct{I},\vct{J}) \mtx{T_C}^{-1}$ 
        \State $\svhat \gets $ estimate of the CUR spectral error by \cref{eq:2normestimate} \Comment{target level}
        \State $\mtx{\widetilde{P}}_\l^{-1} \gets \svhat\mtx{Q_R} \mtx{M}^{-1}\mtx{Q_R}^\top + (\mtx{I_n}-\mtx{Q_R}\mtx{Q_R}^\top)$
        \Comment{apply via \texttt{matvec}}
    \end{algorithmic}
\end{algorithm}


\begin{algorithm}
    \caption{\texttt{augmentQR}}
    \label{alg:augmentqr}
    \begin{algorithmic}[1]
        \Require Orthonormal $\mtx{Q}\in\F^{m\times n}$, upper-triangular $\mtx{R}\in\F^{n\times n}$, additional columns $\mtx{C}\in\F^{m\times p}$
        \Ensure Orthonormal $\mtx{Q}_{\mathrm{new}}\in\F^{m\times(n+p)}$ and upper-triangular $\mtx{R}_{\mathrm{new}}\in\F^{(n+p)\times(n+p)}$ with $[\mtx{Q}_{\mathrm{new}}, \mtx{R}_{\mathrm{new}}]=\texttt{qr}([\mtx{Q}\mtx{R},\mtx{C}])$
        \Statex
        \State $\mtx{P}_1 \gets \mtx{Q}^\top \mtx{C}$
        \State $\mtx{E} \gets \mtx{C} - \mtx{Q}\mtx{P}_1$ \Comment{project and form the residual}
        \State $\mtx{P}_2 \gets \mtx{Q}^\top \mtx{E}$ \Comment{reorthogonalization (optional but recommended)}
        \State $\mtx{E} \gets \mtx{E} - \mtx{Q}\mtx{P}_2$
        \State $[\mtx{Q}_E, \mtx{R}_E] \gets \texttt{qr}(\mtx{E},0)$ \Comment{sketched Cholesky QR can be used}
        \State $\mtx{Q}_{\mathrm{new}} \gets [\mtx{Q},\ \mtx{Q}_E]$
        \State $\mtx{R}_{\mathrm{new}} \gets \begin{bmatrix} \mtx{R} & \mtx{P}_1+\mtx{P}_2 \\[2pt] \mtx{0} & \mtx{R}_E \end{bmatrix}$
        \Statex
        \State \Return $\mtx{Q}_{\mathrm{new}},\ \mtx{R}_{\mathrm{new}}$
    \end{algorithmic}
\end{algorithm}

\subsection{Computational Complexity}\label{app:complexity}

The computational cost of the proposed algorithm is dominated by two primary components: (i) factorizations associated with constructing and updating the low-rank preconditioner, and (ii) matrix-matrix multiplications involving the full problem dimensions $m$ and $n$. We detail each below, emphasizing how these costs are mitigated in practice.

\paragraph{Preconditioner construction}
Given CUR factors, constructing a rank-$\ell$ preconditioner costs $\mathcal{O}((m+n+\ell)\ell^2)$, while both storage and application cost $\mathcal{O}(n\ell)$.

\paragraph{Factorization costs}

A dominant cost arises in each preconditioning update (\cref{alg:preconditioner}, line~5), where an SVD of a dense $\ell \times \ell$ matrix is required. This step costs $\mathcal{O}(\ell^3)$ and is typically the most expensive operation. The SVD-free variant introduced in \cref{sec:svdfreepreconditioner} eliminates this decomposition entirely, reducing the construction cost by $\mathcal{O}(\ell^3)$ (with a large constant). The trade-off is a modest increase in application cost: applying the preconditioner requires two additional triangular solves, contributing approximately $2\ell^2$ extra flops per iteration.

Maintaining a QR factorization of the growing matrix $\mtx{R}^\top \in \F^{n \times q\ell_0}$ (with $q = \ell/\ell_0$) would naively cost $\mathcal{O}(n\ell^3/\ell_0)$. Instead, we update the factorization incrementally using \cref{alg:augmentqr}, reducing the cost to $\mathcal{O}(n\ell^2)$. In addition, we employ a block QR append strategy together with sketched Cholesky QR\cite{balabanov2022randomized}, which improves efficiency while maintaining numerical stability.

The pseudoinverse of the CUR core matrix introduces another factorization cost. A direct QR-based approach leads to $\mathcal{O}(\ell^4)$ complexity. While update techniques based on Givens rotations can reduce this to $\mathcal{O}(\ell^3/\ell_0)$, in practice we instead exploit highly optimized dense factorizations. In particular, replacing QR with LU can improve performance due to better parallel scalability, while remaining sufficiently accurate for moderately conditioned problems.

Finally, LUPP factorizations required in the CUR construction incur a cost of $\mathcal{O}(\ell \ell_0 (m+n))$.

\paragraph{Matrix-matrix multiplications}

Matrix-matrix multiplications involving the full dimensions $m$ and $n$ arise in several parts of the algorithm. In particular, the computation of column and row residuals during CUR updates involves products with nominal complexity $\mathcal{O}((m+n)\ell^2)$. These operations are implemented as BLAS-3 kernels and therefore achieve high efficiency on modern hardware. When $\mtx{C}$ and $\mtx{R}$ are sparse, the effective cost is substantially reduced.

In forming and applying the preconditioner, we avoid explicitly constructing large dense intermediate matrices. For example, products of the form $\mtx{Q_R}\Vhat_M$ are applied implicitly via matrix-vector operations. When $\mtx{R}$ is sparse, we further exploit the representation $\mtx{R}\mtx{T_R^{-1}}$ to reduce both computational and storage costs.

\paragraph{Sketching and additional costs}

The initial sketching step (\cref{alg:aplls}, lines~3--4) can be expensive for large $m$ and $n$. A Gaussian sketch requires $\mathcal{O}(mn\ell_0)$ operations and is therefore impractical at scale. To mitigate this, we employ sparse sign embeddings~\cite{nelson2013osnap, cohen2016nearly}, reducing the cost to $\mathcal{O}(\xi \cdot \mathrm{nnz}(\mtx{A}))$, where $\xi$ is the number of nonzeros per column of the sketching matrix (typically $\xi = \min(\ell_0, 8)$~\cite{tropp2019streaming}).

Additional costs include residual norm evaluation for convergence monitoring, which requires $\mathcal{O}(n\ell_0)$ operations and is negligible compared with the dominant terms above.

\section{Numerical Results}\label{app:numericalstudies_detail}

This section contains the technical details required for an efficient implementation, including a detailed explanation of the interplay between spectral structure and Krylov convergence, sketch choices, and hyperparameter choices.

\subsection{Overpreconditioning}\label{sec:overpreconditioning_detail}

Consider a matrix $\mtx{A}$ whose singular values satisfy $\sigma_1 \geq \cdots \geq \sigma_\ell \gg \sigma_{\ell+1} \geq \cdots \geq \sigma_n$. A rank-$k$ spectral preconditioner with $k=\ell$ collapses the dominant singular values to the level $\hat{\sigma}=\sigma_\ell$, thereby preserving a clear separation from the trailing singular values $\sigma_{\ell+1},\ldots,\sigma_n$. In contrast, when $k > \ell$, the target level $\hat{\sigma}=\sigma_k$ lies within the trailing cluster, destroying the spectral separation between dominant and trailing components.

This spectral gap is precisely what allows LSQR to distinguish directions that are already effectively resolved from those that still require correction. Indeed, LSQR builds Krylov subspaces of the form
\[
\mathcal{K}_j(\widetilde{\mtx{A}}^\top\widetilde{\mtx{A}}, \widetilde{\mtx{A}}^\top \vct{b}), \quad \widetilde{\mtx{A}}^\top \vct{b} = \widetilde{\mtx{V}}\widetilde{\mtx{\Sigma}}\widetilde{\mtx{U}}^\top \vct{b},
\]
where $\widetilde{\mtx{A}}:= \mtx{A}\mtx{P}_{\l,\mu}^{-1}$,
so the early Krylov bases are dominated by singular directions corresponding to large values of $\widetilde{\sigma}_i \widetilde{\vct{u}}_i^\top \vct{b}$. A clear spectral gap allows Krylov polynomials to selectively resolve these dominant components.

When the spectral gap is destroyed by overpreconditioning ($k>\ell$), the singular values of the preconditioned matrix become tightly clustered. As a result, singular directions are no longer well separated spectrally, Ritz vectors mix dominant and trailing components, and Krylov polynomial filtering becomes ineffective. Small perturbations can easily rotate the singular vectors.

In problems where the transformed solution $\vct{y}_\ast=\mtx{P}\vct{x}_\ast$ is approximately uniformly distributed across singular directions (for example, in inconsistent or consistent-$b$ settings), convergence is then driven primarily by the reduced condition number and is less sensitive to the precise alignment of early Krylov bases. In contrast, in consistent-$x$ settings where $\vct{y}_\ast$ is strongly biased toward the dominant right singular subspace of the original matrix, the misalignment between early Krylov bases and these dominant directions leads to inefficient convergence despite improved conditioning.

This analysis explains why overpreconditioning can be counterproductive and directly motivates an adaptive strategy that alternates between progressive preconditioning and LSQR phases. In the early LSQR phases, the solution is sought for a preconditioned problem that retains sufficient spectral separation to enable efficient resolution of the dominant components. In later phases, LSQR is warm-started using the driving vector $\widetilde{\mtx{A}}^\top(\vct{b}-\widetilde{\mtx{A}}\vct{y}_0)$, at which point the presence of a pronounced spectral gap is no longer essential, as the residual is confined to a progressively refined subspace. Moreover, the adaptive CUR approximation inherently prevents overestimation of the target rank, ensuring that additional preconditioning effort is introduced only when it is likely to be beneficial.

\subsection{Sketch comparison}\label{sec:sketch}

We also study the impact of different sketching methods. In this set of experiments, we consider three sketch types: Gaussian embeddings, sparse embeddings with $\xi=8$~\cite{kane2014sparser}, and SRFT (or 1-hashing embeddings~\cite{cartis2021hashing}), all with sketch dimension $s$. \cref{table:sketchcost} summarizes the asymptotic cost of applying the different sketching operators. For Gaussian and sparse embeddings, the cost scales linearly with the number of nonzeros in $A$. In contrast, the SRFT embedding relies on a dense orthogonal transform, resulting in a cost that scales with the full matrix dimensions.

\jh{experiments on sparse matrices}
\input{plot_tex/num_stud_sketch}

\begin{table}[t]
\centering
\caption{Asymptotic cost of applying different sketching matrices to an $m \times n$ matrix $A$.
Here $s$ denotes the sketch dimension, $\xi$ the number of nonzeros per column in the sparse embedding,
and $\mathrm{nnz}(A)$ the number of nonzero entries in $A$.}
\label{table:sketchcost}
\begin{tabular}{lc}
\toprule
Sketch type & Cost \\ 
\midrule
Gaussian embedding 
& $\mathcal{O}\!\left(s\,\mathrm{nnz}(A)\right)$ \\

Sparse embedding 
& $\mathcal{O}\!\left(\xi\,\mathrm{nnz}(A)\right)$ \\

SRFT\footnotemark 
& $\mathcal{O}\!\left(m n \log m\right)$ \\
\bottomrule
\end{tabular}
\end{table}
\footnotetext{The stated cost for the SRFT embedding assumes a standard fast transform
(e.g.\ DCT or Hadamard), which densifies the input matrix.}

\jh{The block size $\ell_0$ here is playing a big role throughout the algorithm. Pritchard et al.~\cite{pritchard2025fast} show that high approximation accuracy is empirically preserved even for block sizes as small as five, but for optimal computational speed, larger block sizes (around 100) perform better.}

\subsection{Hyperparameters for adaptivity}

The behavior of the fully adaptive scheme is controlled by three hyperparameters: the CUR error tolerance $\varepsilon_{cur}$, the re-preconditioning tolerance $\nu_{prec}$, and the dynamic LSQR stopping tolerance $\nu_{lsqr}$. Together, these regulate rank growth, preconditioner updates, and the balance between preconditioning and Krylov iterations.

From \cref{sec:adaptiveprec_sensitivity}, we see that our method is robust to moderate overestimation of the target rank but sensitive to underestimation. Consequently, the CUR tolerance $\varepsilon_{cur}$ must be chosen to avoid truncating significant spectral components. For regularized problems, the regularization parameter $\mu$ provides a natural reference point because it determines the desired solution accuracy. Accordingly, we set
\[
\varepsilon_{cur} \propto \mu,
\]
and study its effect on rank selection and convergence. The resulting behavior shows that early convergence is largely insensitive to the tolerance, whereas overly small tolerances trigger unnecessary rank growth and increase runtime with negligible residual improvement. Hence, we recommend $\varepsilon_{cur} \in [30\mu, 100\mu]$.

The balance between preconditioning and LSQR solving phases is controlled through the choice of re-preconditioning tolerance $\nu_{prec}$ and LSQR dynamic stopping tolerance $\nu_{lsqr}$, as defined in \cref{sec:adaptivitycriteria}. The re-preconditioning tolerance $\nu_{\mathrm{prec}}$ controls when the preconditioner is updated. Since constructing the preconditioner $\mtx{P}$ is relatively expensive, updates should be triggered only when a meaningful improvement in preconditioning quality is expected, unless we use the SVD-free variant. In practice, performance is largely insensitive to the exact choice of $\nu_{\mathrm{prec}}$, and a moderate value (e.g.\ $\nu_{\mathrm{prec}} = 10$) provides a good balance between avoiding unnecessary updates and maintaining effective preconditioning.

The LSQR tolerance $\nu_{\mathrm{lsqr}}$ determines when intermediate LSQR phases terminate. If $\nu_{\mathrm{lsqr}}$ is too small, LSQR may terminate prematurely despite strong convergence, leading to overly frequent re-preconditioning and increased total runtime. Conversely, if $\nu_{\mathrm{lsqr}}$ is too large, LSQR may continue despite slow progress, resulting in wasted iterations. Empirically, values in the range $\nu_{\mathrm{lsqr}} \in [100, 200]$ provide a good trade-off, as illustrated in \cref{fig:num_stud_cvgrate}.

Warm-starting LSQR after each preconditioner update preserves continuity across phases and helps mitigate degradation due to overpreconditioning, but its effectiveness depends on the choice of $\nu_{\mathrm{lsqr}}$. If LSQR phases terminate too early, overpreconditioning may destroy the spectral gap before a sufficiently accurate solution is reached, leading to slow convergence in the final LSQR phase (see \cref{sec:overpreconditioning}).

\input{plot_tex/num_stud_cvgrate}

\subsection{Benchmarking experiments}\label{app:benchmark_general}
\input{plot_tex/appendix_restimes}
\mainalg{} is not always the fastest method to reach a specific tolerance, but it consistently delivers stable and competitive performance across the entire time horizon.

%% file: plot_tex/num_stud_sketch.tex

\begin{figure}[!htbp]
  \centering

  \begin{subfigure}{0.48\textwidth}
    \centering
    \includegraphics[width=\linewidth,height=0.28\textheight,keepaspectratio]
      {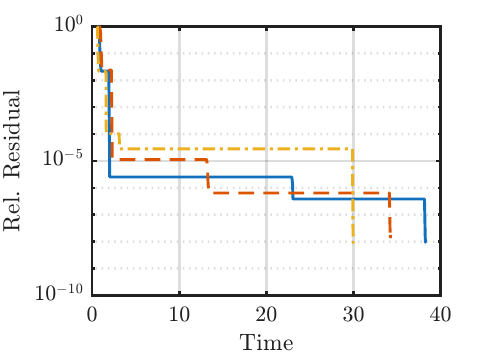}
    \caption{$6000 \times 5000$}
  \end{subfigure}\hfill
  \begin{subfigure}{0.48\textwidth}
    \centering
    \includegraphics[width=\linewidth,height=0.28\textheight,keepaspectratio]
      {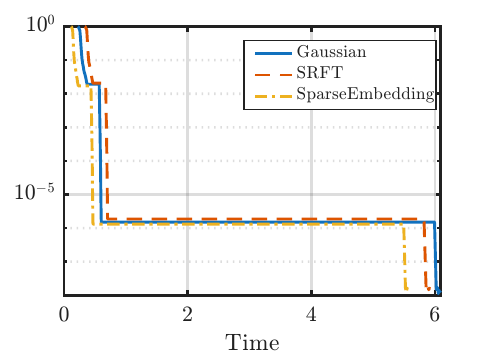}
    \caption{$15000 \times 2000$}
  \end{subfigure}\hfill
  
  \caption{Different Sketches. SRFT dependent on log m, Gaussian dependent on sketch size (aka. block size) is making (a) and (b) different.}
  \label{fig:num_stud_cvgrate1}
\end{figure}

%% file: plot_tex/num_stud_cvgrate.tex

\begin{figure}[!htbp]
  \centering

  \begin{subfigure}{0.48\textwidth}
    \centering
    \includegraphics[width=\linewidth,height=0.28\textheight,keepaspectratio]
      {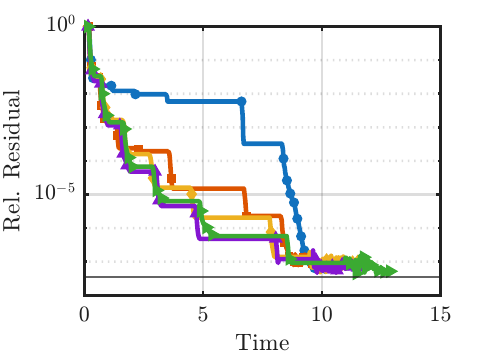}
    \caption{Smooth decay}
  \end{subfigure}\hfill
  \begin{subfigure}{0.48\textwidth}
    \centering
    \includegraphics[width=\linewidth,height=0.28\textheight,keepaspectratio]
      {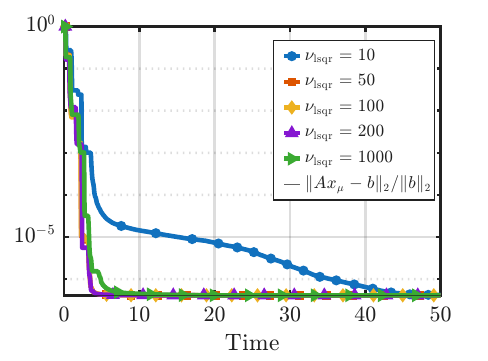}
    \caption{Sharp decay}
  \end{subfigure}\hfill
  
  \caption{Different dynamic stopping criterion. On dense $6000 \times 5000$ coherent matrix with smooth decay and sharp decay.}
  \label{fig:num_stud_cvgrate}
\end{figure}

%% file: plot_tex/appendix_restimes.tex
\begin{figure}[!htbp]
  \centering
  \begin{subfigure}{0.33\textwidth}
    \centering
    \includegraphics[width=\linewidth,height=0.28\textheight,keepaspectratio]
        {figures/restimes32-19Mar/restime_d_6k5k_x_incoh.pdf}
    \caption{Sharp decay\\Incoherent}
  \end{subfigure}\hfill
  \begin{subfigure}{0.33\textwidth}
    \centering
    \includegraphics[width=\linewidth,height=0.28\textheight,keepaspectratio]
      {figures/restimes32-19Mar/restime_d_6k5k_xincon2_incoh.pdf}
    \caption{Sharp decay\\Incoherent}
  \end{subfigure}\hfill
  \begin{subfigure}{0.33\textwidth}
    \centering
    \includegraphics[width=\linewidth,height=0.28\textheight,keepaspectratio]
      {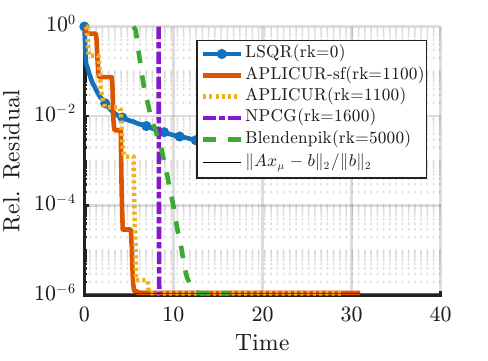}
    \caption{Sharp decay\\Incoherent}
  \end{subfigure}\hfill

  \begin{subfigure}{0.33\textwidth}
    \centering
    \includegraphics[width=\linewidth,height=0.28\textheight,keepaspectratio]
        {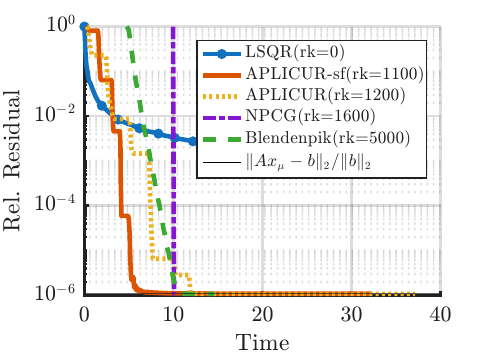}
    \caption{Sharp decay\\Coherent}
  \end{subfigure}\hfill
  \begin{subfigure}{0.33\textwidth}
    \centering
    \includegraphics[width=\linewidth,height=0.28\textheight,keepaspectratio]
      {figures/restimes32-19Mar/restime_d_6k5k_xincon2_coh.pdf}
    \caption{Sharp decay\\Coherent}
  \end{subfigure}\hfill
  \begin{subfigure}{0.33\textwidth}
    \centering
    \includegraphics[width=\linewidth,height=0.28\textheight,keepaspectratio]
      {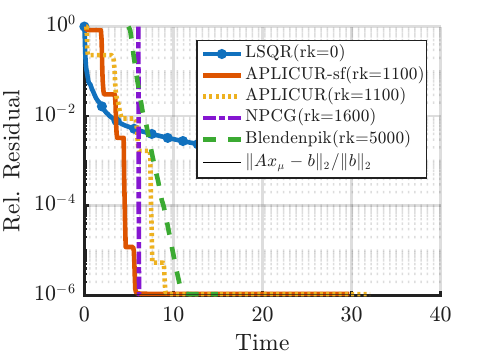}
    \caption{Sharp decay\\Coherent}
  \end{subfigure}\hfill

  \begin{subfigure}{0.33\textwidth}
    \centering
    \includegraphics[width=\linewidth,height=0.28\textheight,keepaspectratio]
        {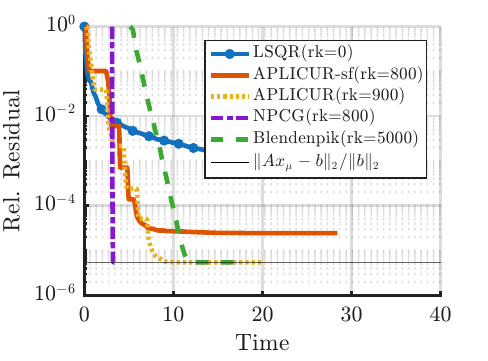}
    \caption{Smooth decay\\Incoherent}
  \end{subfigure}\hfill
  \begin{subfigure}{0.33\textwidth}
    \centering
    \includegraphics[width=\linewidth,height=0.28\textheight,keepaspectratio]
      {figures/restimes32-19Mar/restime_d_6k5k_sm_xincon2_incoh.pdf}
    \caption{Smooth decay\\Incoherent}
  \end{subfigure}\hfill
  \begin{subfigure}{0.33\textwidth}
    \centering
    \includegraphics[width=\linewidth,height=0.28\textheight,keepaspectratio]
      {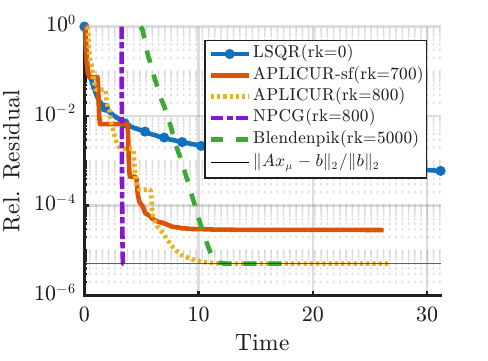}
    \caption{Smooth decay\\Incoherent}
  \end{subfigure}\hfill

  \begin{subfigure}{0.33\textwidth}
    \centering
    \includegraphics[width=\linewidth,height=0.28\textheight,keepaspectratio]
        {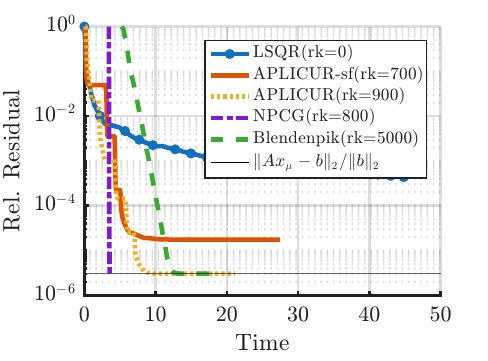}
    \caption{Smooth decay\\Coherent}
  \end{subfigure}\hfill
  \begin{subfigure}{0.33\textwidth}
    \centering
    \includegraphics[width=\linewidth,height=0.28\textheight,keepaspectratio]
      {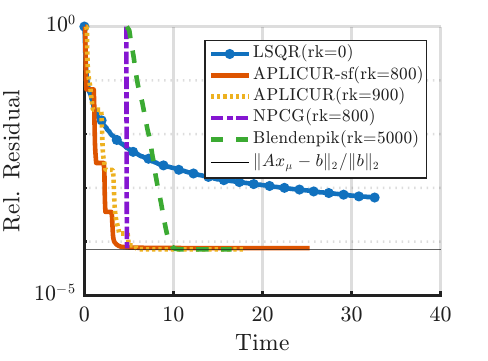}
    \caption{Smooth decay\\Coherent}
  \end{subfigure}\hfill
  \begin{subfigure}{0.33\textwidth}
    \centering
    \includegraphics[width=\linewidth,height=0.28\textheight,keepaspectratio]
      {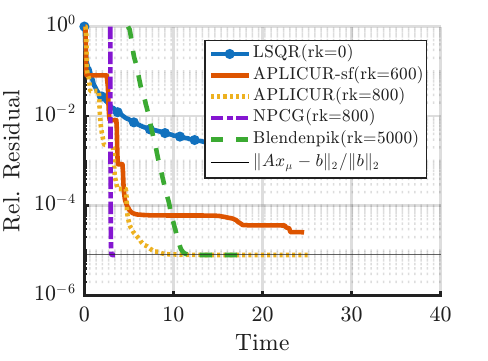}
    \caption{Smooth decay\\Coherent}
  \end{subfigure}\hfill
  
  \caption{Results for $6000 \times 5000$ matrices in the consistent-$\mathbf{x}$ case. Left column: zero noise ($\|\mathbf{e}\|_2 = 0$); middle column: large noise ($\|\mathbf{e}\|_2 = 10^{-2}$); right column: small noise ($\|\mathbf{e}\|_2 = 10^{-12}$).}
\end{figure}

\begin{figure}[!htbp]
  \centering
  
  \begin{subfigure}{0.33\textwidth}
    \centering
    \includegraphics[width=\linewidth,height=0.28\textheight,keepaspectratio]
      {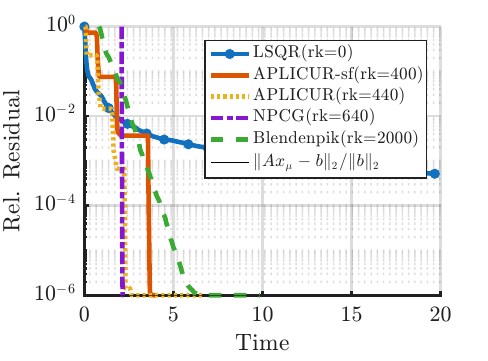}
    \caption{Sharp decay\\Incoherent}
  \end{subfigure}\hfill
  \begin{subfigure}{0.33\textwidth}
    \centering
    \includegraphics[width=\linewidth,height=0.28\textheight,keepaspectratio]
      {figures/restimes32-19Mar/restime_d_15k2k_xincon2_incoh.pdf}
    \caption{Sharp decay\\Incoherent}
  \end{subfigure}\hfill
  \begin{subfigure}{0.33\textwidth}
    \centering
    \includegraphics[width=\linewidth,height=0.28\textheight,keepaspectratio]
      {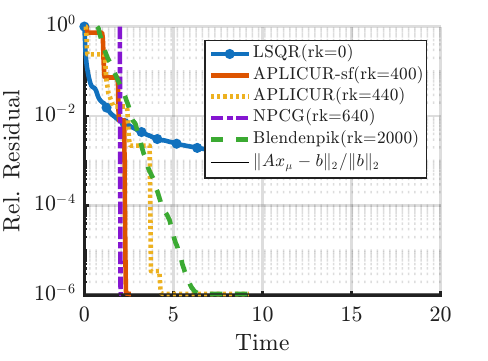}
    \caption{Sharp decay\\Incoherent}
  \end{subfigure}\hfill

  \begin{subfigure}{0.33\textwidth}
    \centering
    \includegraphics[width=\linewidth,height=0.28\textheight,keepaspectratio]
        {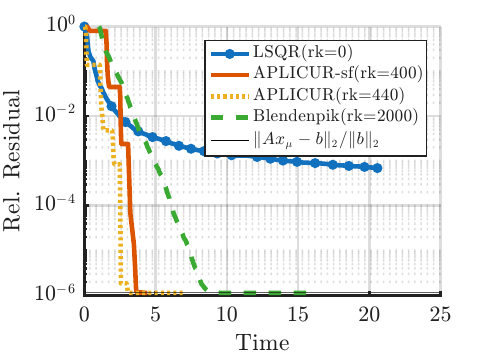}
    \caption{Sharp decay\\Coherent}
  \end{subfigure}\hfill
  \begin{subfigure}{0.33\textwidth}
    \centering
    \includegraphics[width=\linewidth,height=0.28\textheight,keepaspectratio]
      {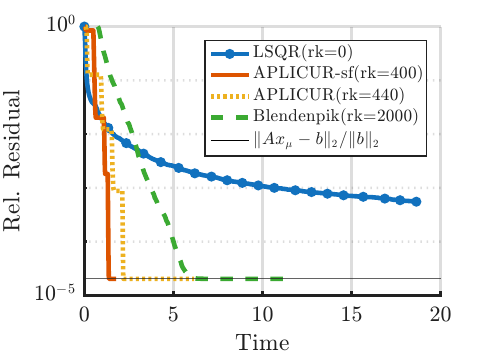}
    \caption{Sharp decay\\Coherent}
  \end{subfigure}\hfill
  \begin{subfigure}{0.33\textwidth}
    \centering
    \includegraphics[width=\linewidth,height=0.28\textheight,keepaspectratio]
      {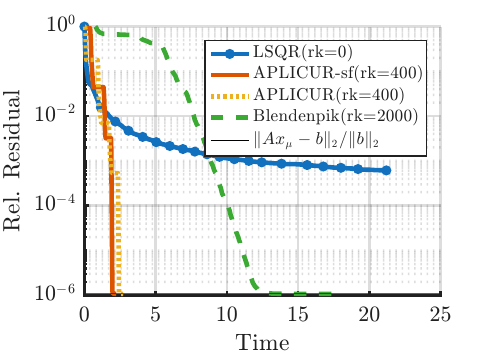}
    \caption{Sharp decay\\Coherent}
  \end{subfigure}\hfill

  \begin{subfigure}{0.33\textwidth}
    \centering
    \includegraphics[width=\linewidth,height=0.28\textheight,keepaspectratio]
        {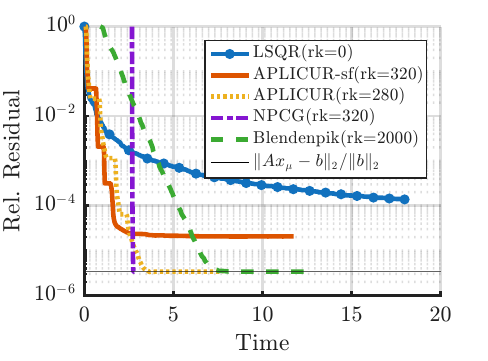}
    \caption{Smooth decay\\Incoherent}
  \end{subfigure}\hfill
  \begin{subfigure}{0.33\textwidth}
    \centering
    \includegraphics[width=\linewidth,height=0.28\textheight,keepaspectratio]
      {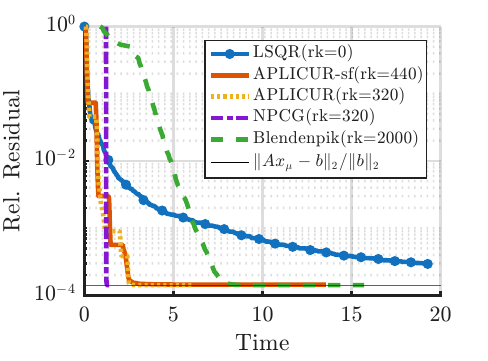}
    \caption{Smooth decay\\Incoherent}
  \end{subfigure}\hfill
  \begin{subfigure}{0.33\textwidth}
    \centering
    \includegraphics[width=\linewidth,height=0.28\textheight,keepaspectratio]
      {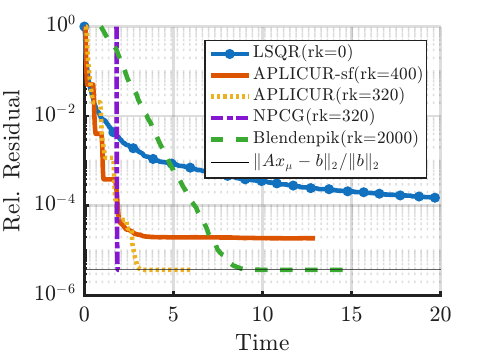}
    \caption{Smooth decay\\Incoherent}
  \end{subfigure}\hfill

  \begin{subfigure}{0.33\textwidth}
    \centering
    \includegraphics[width=\linewidth,height=0.28\textheight,keepaspectratio]
        {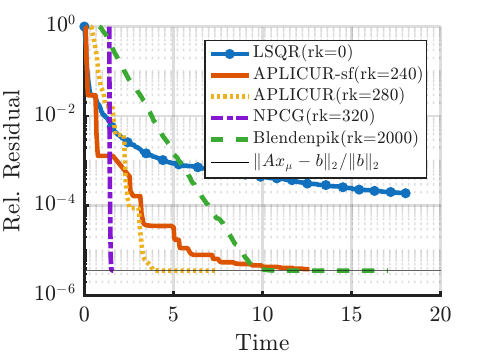}
    \caption{Smooth decay\\Coherent}
  \end{subfigure}\hfill
  \begin{subfigure}{0.33\textwidth}
    \centering
    \includegraphics[width=\linewidth,height=0.28\textheight,keepaspectratio]
      {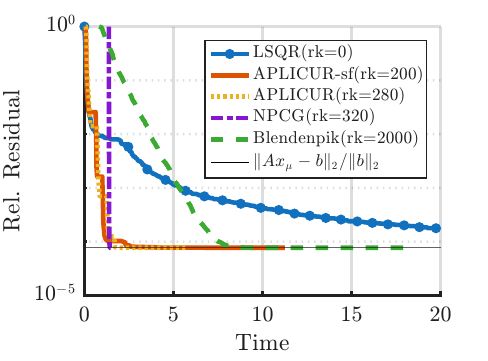}
    \caption{Smooth decay\\Coherent}
  \end{subfigure}\hfill
  \begin{subfigure}{0.33\textwidth}
    \centering
    \includegraphics[width=\linewidth,height=0.28\textheight,keepaspectratio]
      {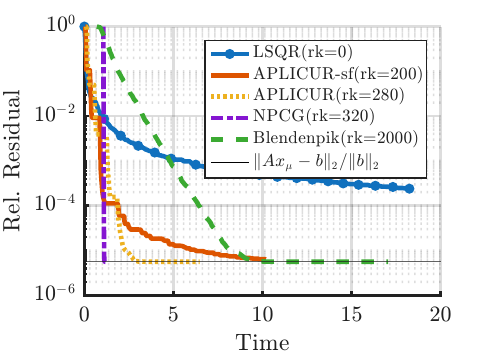}
    \caption{Smooth decay\\Coherent}
  \end{subfigure}\hfill
  
  \caption{Results for $15000 \times 2000$ matrices in the consistent-$\mathbf{x}$ case. Left column: zero noise ($\|\mathbf{e}\|_2 = 0$); middle column: large noise ($\|\mathbf{e}\|_2 = 10^{-2}$); right column: small noise ($\|\mathbf{e}\|_2 = 10^{-12}$).}
\end{figure}